\documentclass[12pt, a4paper]{article} %,draft

\usepackage{amsmath, amsthm, amsfonts, amsbsy}
\usepackage{cite, mathtools, booktabs}

\usepackage{color, graphicx}
\usepackage{subfig}
\usepackage{epstopdf}

\newcommand{\me}{\mathrm{e}}
\newcommand{\mi}{\mathrm{i}}
\newcommand{\dif}{\mathrm{d}}

\renewcommand{\vec}[1]{\boldsymbol{#1}}

\newtheorem{remark}{Remark}[section]
\newtheorem{theorem}{Theorem}[section]

\newtheorem{problem}{Problem}[section]

\numberwithin{equation}{section}
\graphicspath{{Figures/}}

\begin{document}
\title{Solving the multi-frequency electromagnetic inverse source problem by the Fourier method}
\author{Guan Wang \footnote{School of Mathematics, Jilin University, Changchun, P. R. China.   {\em Email: wangguan@neepu.edu.cn}}
\and  Fuming Ma  \footnote{School of Mathematics, Jilin University, Changchun, P. R. China.  {\em Email: mfm@jlu.edu.cn}}
\and Yukun Guo \footnote{Corresponding author. Department of Mathematics, Harbin Institute of Technology, Harbin, P. R. China. {\em Email: ykguo@hit.edu.cn}}
\and Jingzhi Li \footnote{Department of Mathematics, Southern University of Science and Technology, Shenzhen, P. R. China. {\em Email: li.jz@sustc.edu.cn}}
}
\date{}

\maketitle

\begin{abstract}

This work is concerned with an inverse problem of identifying the current source distribution of the time-harmonic Maxwell's equations from multi-frequency measurements. Motivated by the Fourier method for the scalar Helmholtz equation and the polarization vector decomposition,  we propose a novel method for determining the source function in the full vector Maxwell's system. Rigorous mathematical justifications of the method are given and numerical examples are provided to demonstrate the feasibility and effectiveness of the method.\\

\end{abstract}

\section{Introduction}

Inverse source problems arise naturally in various application areas such as biomedical tomography, non-invasive detection, target tracking and antenna synthesis. In recent years there have been tremendous advances in the theoretical understanding and numerical treatment of inverse source problems, see, e.g., \cite{Badia, Devaney, Arridge, invmonk, Bao, Fokas,He, Anastasio,LiuU, Isakov} and the references therein for relevant studies.

The inverse source problems could be posed in the frequency or time domain. The inverse source problem of determining a source in the Helmholtz equation has been studied in \cite{Arridge, Lin, Devaney, HD, nara, Eller, four, BaoLLT}. For inverse source problems for time-harmonic Maxwell's system, we refer to He and Romanov \cite{He}, Ammari et al. \cite{Bao} and Albanese and Monk \cite{invmonk} for the investigation of the localization of brain activities. Rodr\'iguez et al. \cite{Rodriguez} studied the inverse source problem for the eddy current of Maxwell's equations. For recent works on dynamic inverse source problem of imaging the trajectory of a moving point source, we refer to \cite{GuoLLW, GuoLLW2}.

This work is concerned with the inverse source problem of determining the electric current in the time-harmonic Maxwell's system from the multifrequency near-field measurements. For increasing stability analysis concerning the inverse source problem with multi-frequencies, we refer to \cite{ChengIL, LY16, BLZ17}.

A numerical method based on the Fourier expansion of the source has been proposed to solve the multi-frequency inverse source problem for the Helmholtz equation in \cite{four}. This Fourier method has been extended to the far-field cases of acoustic inverse source problem \cite{WangGZL}. The Fourier method is easy to implement with computational efficiency. Our goal in this paper is to extend the Fourier method from the scalar Helmholtz equation to full vector Maxwell's equations. Due to the complexity of the profound structure of Maxwell's equations, both the theoretical analysis and numerical implementation in the current study are radically much more challenging than the scalar counterpart presented in \cite{four, WangGZL}. In particular, one would encounter much more complicated technical difficulties in establishing the stability estimates.

The rest of this paper is organized as follows. In the next section, we introduce the model problem and recall some suitable Sobolev spaces. Section 3 is devoted to approximating the source function by the Fourier expansion, a uniqueness result on the approximate source, and then a numerical method for solving the inverse source problem are presented. The result of error estimate of measurements with noise is derived in Section 4. Finally, several numerical examples are provided to show the effectiveness of our method in Section 5.

%========================================================

\section{Problem setting}

In this section, we shall first give a description of the multifrequency inverse source problem under consideration. Then, some notation of relevant Sobolev spaces will be reviewed as the prerequisites for the theoretical analysis in the subsequent sections.
Throughout this paper,  we adopt the convention of using bold and non-bold fonts for vector and scalar values, respectively.

\subsection{Model problem}

In this paper, we consider the inverse source problem of determining a radiating current density excitation $\vec{J}$ in the time-harmonic Maxwell's equations in $\mathbb{R}^3$
\begin{align}
  &  \nabla \times \vec{E}-\mi k \vec{H}=0, \label{eq:maxwell1} \\
   & \nabla \times \vec{H}+\mi k \vec{E}=\vec{J}, \label{eq:maxwell2}
\end{align}
with the Silver-M\"uller radiation condition
\begin{equation}\label{eq:radiation}
    \lim_{|\vec{x}|\to\infty}( \nabla \times \vec{E}\times \hat{\vec{x}}-\mi k\vec{E}  )=0,
\end{equation}
which holds uniformly for all directions $\hat{\vec{x}}=\vec{x}/|\vec{x}|$, with $\vec{x}=(x_1,x_2,x_3)\in \mathbb{R}^3$. Here $\vec{E}$, $\vec{H}$ and $\vec{J}$ denote, respectively, the electric field, the magnetic field and the current density vector. The wavenumber/frequency of the problem is defined by the positive constant $k=\omega\sqrt{\epsilon_0\mu_0}$, with $\omega$ denoting the angular frequency and $\epsilon_0$ and $\mu_0$ the electric permittivity and the magnetic permeability in vacuum. Eliminating $\vec{H}$ in \eqref{eq:maxwell1}-\eqref{eq:maxwell2} then the time-harmonic Maxwell's equations can be written as
\begin{equation}\label{eq:maxwell}
    \nabla \times \nabla \times \vec{E}-k^{2}\vec{E}=\mi k \vec{J}.
\end{equation}

We consider the multi-frequency inverse source problem with the following assumptions:\\
 (\romannumeral 1) The current density vector $\vec{J}$ is independent of $k$ and
 \begin{equation}\label{eq:supp}
    \vec{J}\in (L^2(\mathbb{R}^{3}))^3,\quad {\rm{supp}} \vec{J}\subset V_0,
 \end{equation}
where $V_0$ ia a cube and $V_0\subset B_R:=\{\vec{x}\in\mathbb{R}^3\mid |\vec{x}|<R\}$. \\
(\romannumeral 2) The current source $\vec{J}$ is expressed in the following form
\begin{equation}\label{eq:Jnon}
    \vec{J}=\vec{p}f+\vec{p}\times \nabla g,
\end{equation}
where $f\in L^2(V_0)$ and $g\in H^1(V_0)$. Here, the polarization vector $\vec{p}$ is assumed to be known and belongs to the following admissible set
\[
    \mathbb{P}:=\{\vec{q}\in\mathbb{R}^3|\ \vec{q}\times\vec{l}\neq \vec{0},\
                  \forall \vec{l}\in \mathbb{Z}^3 \backslash \{\vec{0}\} \}.
\]

In what follows, let $\Gamma_R:=\{\vec{x}\in\mathbb{R}^3\mid |\vec{x}|=R\}$ be the measurement surface. Then the multi-frequency inverse source problem can be stated as follows:
\begin{problem}[Multi-frequency Inverse Source Problem]
Given a fixed polarization vector $\vec{p}\in\mathbb{P}$ and a finite set $\mathbb{K}$ of admissible wavenumbers, reconstruct the source function $\vec{J}$ in the form \eqref{eq:Jnon} from the measured data
\[
 \{\vec{E}(\vec{x}; k, \vec{p})\mid \forall\vec{x}\in \Gamma_R,\ \forall k\in\mathbb{K}\}.
\]
where $\vec{E}(\vec{x}; k, \vec{p})$ indicates the dependence of the radiated field $\vec{E}(\vec{x})$ on the wavenumber $k$ and the polarization vector $\vec{p}$.
\end{problem}
%It is well known that volume sources cannot be uniquely determined from surface measurements due to the existence of non-radiating sources \cite{Bleistein,invmonk}.
%It is derived by Lindell in \cite{tetm} that every source can be expressed in the form
%    \begin{gather}
%    \vec{J}=\vec{p}f(\vec{x})+\vec{p}\times \nabla g(\vec{x})+\vec{J}^{NR},\notag
%    \end{gather}
%with a unit constant polarization $\vec{p}$, and the last term $\vec{J}^{NR}$ is a nonradiating current
%which produces null field outside of $V_0$.
%In this paper, we make another assumption: \\
%(3) The source is expressed in the form
%    \begin{gather}
%    \vec{J}=\vec{p}f+\vec{p}\times \nabla g,\label{eq:Jnon}
%    \end{gather}
%where $\vec{p}$ is a given constant polarization vector and $f\in L^2(\mathbb{R}^{3}), g\in H^1(\mathbb{R}^{3})$.
%\begin{remark}
%The literature \cite{tetm} shows that bounded electromagnetic sources can be decomposed into a nonradiating part and two parts radiating transverse electric and transverse magnetic fields with respect to a given constant direction.
%\end{remark}

In this paper, we derive a novel and effective numerical method for solving the multi-frequency inverse source problem. Our method is based on the Fourier expansion of the source function $\vec{J}$.  We prove that a finite number of Fourier coefficients can be uniquely determined from the measured fields  corresponding to a set of wavenumbers which are properly chosen. Meanwhile, we establish the explicit formula for each coefficient.

%----------------------------------------------------------------------------------------

\subsection{Tangential vector spaces}\label{sec:sobolev_space}

In this subsection, we recall some essential ingredients for tangential vector Sobolev spaces and we refer to \cite{monk} for more relevant details.

Let $\Gamma\subset\mathbb{R}^3$ be a generic smooth and closed surface. Then the space of tangential vector fields is defined by
\[
    L_t^2({\Gamma}):=\{\vec{u}\in(L^2({\Gamma}))^3\mid \vec{u}\cdot \vec{\nu}=0\ {\rm{on}}\ \Gamma\},
\]
where $\vec{\nu}$ denotes the unit outward normal of $\Gamma$.

Let $\mathbb{S}^2:=\{\vec{x}\in\mathbb{R}^3\mid\ |\vec{x}|=1\}$ be the unit sphere and $\{\vec{e}_{r},\ \vec{e}_{\theta},\ \vec{e}_{\varphi} \}$ be the unit vectors of the spherical
coordinates, where $\vec{e}_{r}=\hat{\vec{x}}$, and $\{Y_{n}^{m}(\hat{\vec{x}}):\ m=-n,\cdots,n,\ n=0,1,2,\cdots\}$ are the spherical harmonics.

Denote the vector spherical harmonics
\[
   \vec{U}_{n}^{m}(\hat{\vec{x}})=\frac{1}{\sqrt {n(n+1)}}\nabla _{\mathbb{S}^2}
       Y_{n}^{m}(\hat{\vec{x}}) , \quad \vec{V}_{n}^{m}(\hat{\vec{x}})
       =\hat{\vec{x}}\times  \vec{U}_{n}^{m}(\hat{\vec{x}}),\quad \hat{\vec{x}}\in\mathbb{S}^2,
\]
where the surface gradient $\nabla _{\mathbb{S}^2}$ is defined by
\[
    \nabla _{\mathbb{S}^2}:=\frac{\partial}{\partial \theta}\vec{e}_{\theta}
    +\frac{1}{\sin{\theta}} \frac{\partial}{\partial \varphi}\vec{e}_{\varphi}.
\]
Then the set of all vector spherical harmonics $\{\vec{U}_{n}^{m},\ \vec{V}_{n}^{m}:\ m=-n,\cdots,n,\ n=1,2,\cdots\}$ form a complete orthonormal basis of $L^{2}_{t}(\mathbb{S}^2)$.

Now we introduce the space of three-dimensional vector functions by
\[
   H({\rm{curl}};B_R):=\{\vec{u}\in (L^2(B_R))^3\mid\nabla\times\vec{u}\in (L^2(B_R))^3\},
\]
and its trace space $H^{-1/2}({\rm{Div}};\Gamma_R)$.
We suppose $\vec{\zeta}\in L_t^2({\Gamma_R}) \cap H^{-1/2}({\rm{Div}};\Gamma_R)$ has representation
\[
    \vec{\zeta}(\vec{x})= \sum_{n=1}^{\infty}\sum_{m=-n}^{n}\Big(\alpha_{n,m}
         \vec{U}_n^m(\vec{\hat{x}}) +\beta_{n,m}\vec{V}_n^m(\vec{\hat{x}})\Big),
\]
then its norms in $L_t^2({\Gamma_R})$ and $H^{-1/2}({\rm{Div}};\Gamma_R)$ are defined by
\begin{align*}
   &\|\vec{\zeta}\|_{L_t^2({\Gamma_R})}:=\left(\sum_{n=1}^{\infty}\sum_{m=-n}^{n}
       \left(|\alpha_{n,m}|^2+|\beta_{n,m}|^2\right)\right)^{1/2},\\
   &\|\vec{\zeta}\|_{H^{-1/2}({\rm{Div}},\Gamma_R)}:=\left(\sum_{n=1}^{\infty}\sum_{m=-n}^n\left(\sqrt{n(n+1)}|\alpha_{n,m}|^{2}+\frac{|\beta_{n,m}|^2}{\sqrt{n(n+1)}}\right)\right)^{1/2}.
\end{align*}

In order to compute the data $\hat{\vec{x}}\times \nabla\times\vec{E}$ from $\hat{\vec{x}}\times \vec{E}$, we need the electric-to-magnetic Calder\'{o}n operator $G:H^{-1/2}({\rm{Div}};\Gamma_R)\to H^{-1/2}({\rm{Div}};\Gamma_R)$ which is defined by
\begin{equation}\label{eq:Calderon}
   G\vec{\zeta}:=\sum_{n=1}^{\infty}\sum_{m=-n}^n
       \left(k^2R\frac{h_{n}^{(1)}(kR)}{z_{n}^{(1)}(kR)}\beta_{n,m} \vec{U}_n^m
        +\frac{1}{R}\frac{z_{n}^{(1)}(kR)}{h_{n}^{(1)}(kR)}\alpha_{n,m}\vec{V}_n^m\right),
\end{equation}
where $h_n^{(1)}$ is the spherical Hankel function of the first kind of order $n$, and $z_n^{(1)}(t)=h_n^{(1)}(t)+th_n^{(1)\prime}(t)$.
%------------------------------------------------------------------------
\section{Fourier approximation and uniqueness}

In this section, we will approximate the functions $\vec{J}$ by Fourier expansions and establish explicit formulas for the Fourier coefficients.

We begin this section with some notation and definitions that are used in the paper. Without loss of generality, let
\[
    V_0:=\left(-\frac{L}{2},\frac{L}{2}\right)^3,
\]
and introduce the following Fourier basis functions
\[
   \phi_{\vec{l}}(\vec{x}):=\exp\left(\mi \frac{2\pi}{L}\vec{l}\cdot\vec{x}\right),\quad \vec{l}\in \mathbb{Z}^{3},
\]
then the functions $f\in L^2(V_0)$ and $g\in H^1(V_0)$ can be written as Fourier expansions of the forms
\[
    f=\sum_{|\vec{l}|\geq 0} a_{\vec{l}}\phi_{\vec{l}},\quad
    g=\sum_{|\vec{l}|\geq 1} b_{\vec{l}}\phi_{\vec{l}},
\]
with the Fourier coefficients
\begin{align}
    a_{\vec{l}}=\frac{1}{L^3}\int_{V_0}f(\vec{x})\overline{\phi_{\vec{l}}(\vec{x})}\,\dif \vec{x},
         \quad |\vec{l}|\geq 0, \label{eq:al}\\
    b_{\vec{l}}=\frac{1}{L^3}\int_{V_0}g(\vec{x})\overline{\phi_{\vec{l}}(\vec{x})}\,\dif \vec{x}, \quad |\vec{l}|\geq 1, \label{eq:bl}
\end{align}
where the overbar denotes the complex conjugate. Thus the source $\vec{J}$ has the following form

\begin{equation}\label{eq:Jexpan}
    \vec{J}=\vec{p}f+\vec{p}\times \nabla g
                =a_0\vec{p}+\sum_{|\vec{l}|\geq 1} \left(a_{\vec{l}}\vec{p}+
                \frac{2\pi\mi}{L} b_{\vec{l}}(\vec{p}\times \vec{l})\right)\phi_{\vec{l}}.
\end{equation}
We introduce the Sobolev space of order $\sigma>0$
\begin{align*}
(H^{\sigma}_{\vec{p}}(V_0))^3:=\{&\vec{\Phi}\in(H^{\sigma}(V_0))^3\mid
              \vec{\Phi}=\vec{p}f+\vec{p}\times \nabla g, \\
              & f\in H^{\sigma}(V_0),\ g\in H^{\sigma+1}(V_0),\ \vec{p}\in\mathbb{S}^2 \},
\end{align*}
%then we equip $(H^{\sigma}_{\vec{p}}(V_0))^3$ 
with the norm
\begin{equation}\label{eq:Jnorm}
\|\vec{\Phi}\|_{\vec{p},\sigma}=\left( L^3\sum_{|\vec{l}|\geq0}(1+|\vec{l}|^2)^{\sigma}|a_{\vec{l}}|^2+4\pi^2L
\sum_{|\vec{l}|\geq1}(1+|\vec{l}|^2)^{\sigma}|\vec{p}\times \vec{l}|^2|b_{\vec{l}}|^2\right)^{1/2},
\end{equation}
where $\vec{\Phi}\in  (H^{\sigma}_{\vec{p}}(V_0))^3$ has the Fourier expansion of the form
\[
\vec{\Phi}=a_0\vec{p}+\sum_{|\vec{l}|\geq1} \left(a_{\vec{l}}\vec{p}+
                \frac{2\pi\mi}{L} b_{\vec{l}}(\vec{p}\times \vec{l})\right)\phi_{\vec{l}}.
\]
%It is easy to check that this norm is nothing but the norm in Sobolev space of periodic functions and especially $\|\vec{\Phi}\|_{(L^2(V_0))^3}=\|\vec{\Phi}\|_{\vec{p},0}$.

Now we introduce the definition of admissible polarization vectors in the source function (\ref{eq:Jnon}).

%\begin{definition}[Admissible polarization]
%The set of admissible polarizations is defined by
%\begin{equation}
%    \mathbb{P}:=\{\vec{p}\in\mathbb{S}^2|\ \vec{p}\times\vec{l}\neq \vec{0},\
%                  \forall \vec{l}\in \mathbb{Z}^3 \backslash  \{\vec{0}\} \}.
%\end{equation}
%\end{definition}

For an admissible polarization vector $\vec{p}\in\mathbb{P}$, the decomposition
\begin{align}
    \vec{p}&=\frac{\vec{p}\cdot\vec{l}}{|\vec{l}|^2}\vec{l}
    +\frac{\vec{p}\cdot\vec{l}\times(\vec{p}\times\vec{l})}{|\vec{l}\times(\vec{p}\times\vec{l})|^2}
    \vec{l}\times(\vec{p}\times\vec{l})\notag\\
    &=\frac{\vec{p}\cdot\vec{l}}{|\vec{l}|^2}\vec{l}+\frac{1}{|\vec{l}|^2}\vec{l}
    \times(\vec{p}\times\vec{l})\label{eq:decu}
\end{align}
is crucial for computing the Fourier coefficients. For simplicity, we denote $\vec{v}_{\vec{l}}=\vec{p}\times\vec{l},\ \vec{w}_{\vec{l}}=\vec{l}\times(\vec{p}\times\vec{l})$, and it is clear that $\{\vec{l},\vec{v}_{\vec{l}},\vec{w}_{\vec{l}}\}$ form an orthogonal basis of $\mathbb{R}^3$. A direct computation shows that
\begin{align}
    \nabla\times\nabla\times(\vec{u}_{\vec{l}}\phi_{\vec{l}})
    =-\Delta(\vec{u}_{\vec{l}}\phi_{\vec{l}})+\nabla\nabla\cdot(\vec{u}_{\vec{l}}\phi_{\vec{l}})
    =\frac{4\pi^2}{L^2}|\vec{l}|^2 \vec{u}_{\vec{l}}\phi_{\vec{l}},\label{eq:k1}
\end{align}
where $\vec{u}_{\vec{l}}=\vec{v}_{\vec{l}}$ or $\vec{w}_{\vec{l}}$.
\par
Combining $\eqref{eq:Jexpan}$ and $\eqref{eq:decu}$, we obtain
    \begin{align}
    \vec{J}=&\vec{p}f+\vec{p}\times \nabla g\notag\\
           =&a_{\vec{0}}\vec{p}
           +\sum_{|\vec{l}|\geq1}\left( \frac{a_{\vec{l}}}{|\vec{l}|^2}
                \left((\vec{p}\cdot\vec{l})\vec{l}+\vec{w}_{\vec{l}}\right)
           +\frac{2\pi\mi}{L}b_{\vec{l}}\vec{v}_{\vec{l}}\right)\phi_{\vec{l}}.\label{eq:Jnew}
    \end{align}
We can approximate $\vec{J}$ by the truncated Fourier expansion
\begin{align}
    \vec{J_N}=a_{\vec{0}}\vec{p}
              +\sum_{1\leq|\vec{l}|\leq N}\left( \frac{a_{\vec{l}}}{|\vec{l}|^2}
                \left((\vec{p}\cdot\vec{l})\vec{l}+\vec{w}_{\vec{l}}\right)
              +\frac{2\pi\mi}{L}b_{\vec{l}}\vec{v}_{\vec{l}}\right)\phi_{\vec{l}}.\label{eq:JN}
\end{align}
Then we have the following approximation result.
\begin{theorem}\label{th:JJNerr}
Let $\vec{J}$ be a function in $(H^{\sigma}_{\vec{p}}(V_0))^3, \sigma>0$, then the following estimate holds
\begin{equation}\label{eq:N_estimate}
    \|\vec{J}-\vec{J}_N\|_{\vec{p},\mu}\leq N^{\mu-\sigma} \|\vec{J}\|_{\vec{p},\sigma},\quad 0\leq\mu\leq\sigma.
\end{equation}
\end{theorem}

\begin{proof}
From $\eqref{eq:Jnorm},\eqref{eq:Jnew}$ and $\eqref{eq:JN}$, we obtain
\begin{align}
    &\|\vec{J}-\vec{J}_N\|^2_{\vec{p},\mu}\notag\\
    =&\sum_{|\vec{l}|> N}\left(1+|\vec{l}|^2\right)^{\mu}
      \left( L^3|a_{\vec{l}}|^2+4\pi^2L|\vec{p}\times \vec{l}|^2|b_{\vec{l}}|^2\right)\notag\\
    \leq&\frac{1}{\left(1+N^2\right)^{\sigma-\mu}}\sum_{|\vec{l}|> N}
      \left(1+|\vec{l}|^2\right)^{\sigma} \left( L^3|a_{\vec{l}}|^2+4\pi^2L|\vec{p}
      \times \vec{l}|^2|b_{\vec{l}}|^2\right)\notag\\
    \leq&\frac{1}{\left(1+N^2\right)^{\sigma-\mu}} \|\vec{J}\|_{\vec{p},\sigma}^2,\notag
\end{align}
which leads to the estimate $\eqref{eq:N_estimate}$.
\end{proof}

Motivated by Theorem \ref{th:JJNerr}, the inverse source problem under concern is to determine an approximate source $\vec{J}_N$  in the form (\ref{eq:JN}) with a fixed  admissible polarization vector $\vec{p}\in \mathbb{P}$ from the measurements $\{\hat{\vec{x}}\times \vec{E}(\cdot;k_j, \vec{p})\}$ on $\Gamma_R$ at a finite set of wavenumbers $\{k_j\}$. To this end, we will establish explicit formulas for the Fourier coefficients $a_{\vec{l}}$ and $b_{\vec{l}}$. Before we show that, we first state a uniqueness result.

\begin{theorem}\label{th:ab}
Let $\vec{p}\in \mathbb{P}$ and $\mathbb{K}_N:=\{2\pi |\vec{l}|/L\mid \vec{l}\in \mathbb{Z}^{3}, 1\leq |\vec{l}| \leq N \}$, then the Fourier coefficients $\{a_{\vec{l}}\}_{1\leq|\vec{l}|\leq N}$ and $\{b_{\vec{l}}\}_{1\leq|\vec{l}|\leq N}$ of $\vec{J}$ in $\eqref{eq:al}$ and $\eqref{eq:bl}$ can be uniquely determined by the measurements $\{\vec{E}(\vec{x}; k, \vec{p})\mid \vec{x}\in \Gamma_R, k\in \mathbb{K}_N\}$.
\end{theorem}

\begin{proof}
Let $\vec{J}$ be a source that produces the data $\{\hat{\vec{x}}\times \vec{E}(\vec{x}; k, \vec{p})\mid \vec{x}\in \Gamma_R, k\in \mathbb{K}_N, \vec{p}\in\mathbb{P}\}$. Using the electric-to-magnetic Calder\'{o}n operator \cite[p.249]{monk}, we know that the source also produces the data $\{\hat{\vec{x}}\times \nabla\times\vec{E}(\vec{x}; k, \vec{p})\mid \vec{x}\in \Gamma_R, k\in \mathbb{K}_N\}$ .

For every $\vec{l}(1\leq |\vec{l}| \leq N)$ and $k=2\pi |\vec{l}|/L$, multiplying equation $\eqref{eq:maxwell}$ by $\vec{v}_{\vec{l}}\overline{\phi_{\vec{l}}(\vec{x})}$ and $\vec{w}_{\vec{l}}\overline{\phi_{\vec{l}}(\vec{x})}$ and integrating over $V_0$, then using $\eqref{eq:k1}$, we obtain
\begin{align*}
     &\mi k\int_{V_0}\vec{J}(\vec{x})\cdot \vec{v}_{\vec{l}}\overline{\phi_{\vec{l}}(\vec{x})}\,\dif \vec{x} \\
    =&\int_{B_R}( \nabla \times \nabla \times \vec{E}(\vec{x})-k^{2}\vec{E}(\vec{x}))\cdot \vec{v}_{\vec{l}}\overline{\phi_{\vec{l}}(\vec{x})}\,\dif \vec{x} \\
    =&\int_{B_R}\vec{E}(\vec{x})\cdot \left(\nabla \times \nabla \times( \vec{v}_{\vec{l}}\overline{\phi_{\vec{l}}(\vec{x})} ) -k^2\vec{v}_{\vec{l}}\overline{\phi_{\vec{l}}(\vec{x})}\right)\dif \vec{x} \\
    &+\int_{\Gamma_R} \left( \hat{\vec{x}}\times \nabla \times \vec{E}(\vec{x}) \cdot (\vec{v}_{\vec{l}}\overline{\phi_{\vec{l}}(\vec{x})})
    +\hat{\vec{x}}\times  \vec{E}(\vec{x}) \cdot \nabla\times(\vec{v}_{\vec{l}}\overline{\phi_{\vec{l}}(\vec{x})})\right)\dif s(\vec{x}) \\
    =&\int_{\Gamma_R} \left( \hat{\vec{x}}\times \nabla \times \vec{E}(\vec{x}) \cdot (\vec{v}_{\vec{l}}\overline{\phi_{\vec{l}}(\vec{x})})
    +\hat{\vec{x}}\times  \vec{E}(\vec{x}) \cdot \nabla\times(\vec{v}_{\vec{l}}\overline{\phi_{\vec{l}}(\vec{x})}) \right)\dif s(\vec{x}),
\end{align*}
and
\begin{align*}
     &\mi k\int_{V_0}\vec{J}(\vec{x})\cdot \vec{w}_{\vec{l}}\overline{\phi_{\vec{l}}(\vec{x})}\,\dif \vec{x} \\
    =&\int_{\Gamma_R} \left( \hat{\vec{x}}\times \nabla \times \vec{E}(\vec{x}) \cdot (\vec{w}_{\vec{l}}\overline{\phi_{\vec{l}}(\vec{x})})
    +\hat{\vec{x}}\times  \vec{E}(\vec{x}) \cdot \nabla\times(\vec{w}_{\vec{l}}\overline{\phi_{\vec{l}}(\vec{x})})\right )\dif s(\vec{x}).
\end{align*}
On the other hand, using $\eqref{eq:Jexpan}$, we have
\begin{align*}
     &\int_{V_0}\vec{J}(\vec{x})\cdot (\vec{w}_{\vec{l}}\overline{\phi_{\vec{l}}(\vec{x})})\,\dif \vec{x}=L^3 |\vec{w}_{\vec{l}}|^2 a_{\vec{l}},\\
     &\int_{V_0}\vec{J}(\vec{x})\cdot (\vec{v}_{\vec{l}}\overline{\phi_{\vec{l}}(\vec{x})})\,\dif \vec{x}=2\pi\mi L^2 |\vec{v}_{\vec{l}}|^2 b_{\vec{l}}.
\end{align*}
Then, we obtain
\begin{align}
     a_{\vec{l}}=&\frac{1}{\mi k|\vec{v}_{\vec{l}}|^2 L^3}
     \int_{\Gamma_R} \bigg( \hat{\vec{x}}\times \nabla \times \vec{E}(\vec{x}) \cdot (\vec{w}_{\vec{l}}\overline{\phi_{\vec{l}}(\vec{x})}) \notag \\
     &+\hat{\vec{x}}\times\vec{E}(\vec{x}) \cdot \nabla\times(\vec{w}_{\vec{l}}\overline{\phi_{\vec{l}}(\vec{x})}) \bigg) \dif s(\vec{x}), \label{eq:al_form} \\
     b_{\vec{l}}=&-\frac{1}{2\pi k L^2 |\vec{v}_{\vec{l}}|^2}
     \int_{\Gamma_R} \bigg( \hat{\vec{x}}\times \nabla \times \vec{E}(\vec{x}) \cdot (\vec{v}_{\vec{l}}\overline{\phi_{\vec{l}}(\vec{x})})   \notag\\
    &+\hat{\vec{x}}\times \vec{E}(\vec{x}) \cdot \nabla\times(\vec{v}_{\vec{l}}\overline{\phi_{\vec{l}}(\vec{x})}) \bigg)\dif s(\vec{x}), \label{eq:bl_form}
\end{align}
then the Fourier coefficients $\{a_{\vec{l}}\}_{1\leq|\vec{l}|\leq N}$ and $\{b_{\vec{l}}\}_{1\leq|\vec{l}|\leq N}$ of $\vec{J}$ are uniquely determined. This completes the proof.
\end{proof}
\par
It is clear that we can reconstruct all Fourier coefficients satisfying $1\leq|\vec{l}|\leq N$ from the measurements on the wavenumbers in $\mathbb{K}_N$ by $\eqref{eq:al_form}$ and $\eqref{eq:bl_form}$. Now we are in a position to introduce the numerical algorithm to reconstruct $a_{\vec{0}}$.
\begin{theorem}\label{th:a0error}
Let $\lambda<1/2$ be a small positive constant. %such that
%\[
%    \lambda<\min\left\{\frac{1}{2},\ \frac{L}{4\pi}\right\}.
%\]
Take $\vec{l}^{*}=(\lambda,0,0)^{\top}$,  $\vec{w}^{*}=\vec{l}^{*}\times(\vec{p}\times\vec{l}^{*})$, $k^*=2\pi\lambda/L$ and
\[
  \phi_{\vec{l}^*}(\vec{x}):=\exp{\left(\mi \frac{2\pi}{L}\vec{l}^*\cdot\vec{x}\right)}.
\]
Let  $a_{\vec{0}}^N$ be defined as
\begin{align}
    a_{\vec{0}}^N:= &\frac{\lambda\pi}{\sin{\lambda\pi}} \bigg(\frac{1}{\mi k^*L^3|\vec{p}\times\vec{l}^*|^2}
    \int_{\Gamma_R} \Big( \hat{\vec{x}}\times \nabla \times \vec{E} \cdot (\vec{w}^*\overline{\phi_{\vec{l}^*}(\vec{x})})\notag \\
    &+\hat{\vec{x}}\times  \vec{E} \cdot \nabla\times(\vec{w}^*\overline{\phi_{\vec{l}^*}(\vec{x})}) \Big)\dif s(\vec{x})
    -\sum^{N}_{|j|=1}\frac{\sin{(j-\lambda)\pi}}{(j-\lambda)\pi}a_{j}\bigg),\label{eq:a0N}
\end{align}
where $a_{j}^{\delta}=a_{\vec{l}}^{\delta}$ for $\vec{l}=(j,0,0)^{\top}$.
%where $a_{j}$ is fixed to be $a_{\vec{l}}$ where $\vec{l}=(j,0,0)^{\top}$.
Then the following estimate holds
\begin{align}
    |a_{\vec{0}}-a_{\vec{0}}^N|\leq \frac{2\lambda}{\sqrt{N}L^{3/2}}\|\vec{J}\|_{\vec{p},0}.\label{eq:a0a0N}
    \end{align}
\end{theorem}

\begin{proof}
It is easy to check that for $\vec{l}\in \mathbb{Z}^3\backslash\{\vec{0}\}$, we have
\begin{equation}\label{eq:phill}
    \int_{V_0} \phi_{\vec{l}}(\vec{x})\overline{\phi_{\vec{l}^*}(\vec{x})}\,\dif \vec{x}=
    \begin{cases}
    \dfrac{L^3\sin(j-\lambda)\pi}{(j-\lambda)\pi}, &\vec{l}=(j,0,0)^{\top},\\
    0,                                                &{\rm{otherwise}},
    \end{cases}
\end{equation}
and
\begin{equation}\label{eq:k3}
    \nabla\times\nabla\times(\vec{w}^*\phi_{\vec{l}^*})
    =\frac{4\pi^2}{L^2}|\vec{l}^*|^2 \vec{w}^*\phi_{\vec{l}^*}.
\end{equation}
Similar to the proof of Theorem \ref{th:ab}, multiplying equation $\eqref{eq:maxwell}$ by
$\vec{w^*}\overline{\phi_{\vec{l^*}}(\vec{x})}$ and integrating over $B_R$,
then using $\eqref{eq:k3}$, we obtain
\begin{align*}
   &\int_{V_0}\vec{J}(\vec{x})\cdot \vec{w}^*\overline{\phi_{\vec{l}^*}(\vec{x})}\,\dif \vec{x}\\
 =&\frac{1}{\mi k^*}\int_{\Gamma_R} \left( \hat{\vec{x}}\times \nabla \times \vec{E}(\vec{x}) \cdot (\vec{w}^*\overline{\phi_{\vec{l}^*}(\vec{x})})
    +\hat{\vec{x}}\times  \vec{E}(\vec{x}) \cdot \nabla\times(\vec{w}^*\overline{\phi_{\vec{l}^*}(\vec{x})}) \right)\dif s(\vec{x}).
\end{align*}
On the other hand, using $\eqref{eq:Jnew}$ and $\eqref{eq:phill}$, we have
\begin{align*}
     &\int_{V_0}\vec{J}(\vec{x})\cdot \vec{w}^*\overline{\phi_{\vec{l}^*}(\vec{x})}\dif \vec{x}\\
     =&L^3|\vec{p}\times\vec{l}^*|^2\frac{\sin{\lambda\pi}}{\lambda\pi}a_{\vec{0}}
     + \sum_{|j|=1}^{\infty}L^3|\vec{p}\times\vec{l}^*|^2
           \frac{\sin(j-\lambda)\pi}{(j-\lambda)\pi}a_j,
\end{align*}
Hence we have the following formula for $a_{\vec{0}}$
\begin{align}
    a_{\vec{0}}= &\frac{\lambda\pi}{\sin\lambda\pi} \bigg(\frac{1}{\mi k^*L^3|\vec{p}\times\vec{l}^*|^2}
    \int_{\Gamma_R} \Big( \hat{\vec{x}}\times \nabla \times \vec{E}(\vec{x}) \cdot (\vec{w}^*\overline{\phi_{\vec{l}^*}(\vec{x})})\notag \\
    &+\hat{\vec{x}}\times\vec{E}(\vec{x}) \cdot \nabla\times(\vec{w}^*\overline{\phi_{\vec{l}^*}(\vec{x})}) \Big)\dif s(\vec{x})
    -\sum^{\infty}_{|j|=1}\frac{\sin{(j-\lambda)\pi}}{(j-\lambda)\pi}a_{j}\bigg),\label{eq:a0}
\end{align}
and the approximate formula for $a_{\vec{0}}$
\begin{align}
    a_{\vec{0}}^N= &\frac{\lambda\pi}{\sin{\lambda\pi}}
    \bigg(\frac{1}{\mi k^*L^3|\vec{p}\times\vec{l}^*|^2}
    \int_{\Gamma_R} \Big( \hat{\vec{x}}\times \nabla \times \vec{E}(\vec{x}) \cdot (\vec{w}^*\overline{\phi_{\vec{l}^*}(\vec{x})})\notag \\
    &+\hat{\vec{x}}\times  \vec{E}(\vec{x}) \cdot \nabla\times(\vec{w}^*\overline{\phi_{\vec{l}^*}(\vec{x})}) \Big)\dif s(\vec{x})
    -\sum^{N}_{|j|=1}\frac{\sin{(j-\lambda)\pi}}{(j-\lambda)\pi}a_{j}\bigg).\notag
\end{align}
Furthermore,
\begin{align}
    |a_{\vec{0}}-a_{\vec{0}}^N|=&\frac{\lambda\pi}{\sin{\lambda\pi}}\sum^{\infty}_{|j|=N+1}
       \bigg|  \frac{\sin{(j-\lambda)\pi}}{(j-\lambda)\pi}a_{j}  \bigg |\notag\\
    =&\sum^{\infty}_{|j|=N+1}\bigg|\frac{\lambda}{j-\lambda}a_{j}  \bigg |\notag\\
    \leq&\left( \sum^{\infty}_{|j|=N+1}\bigg|\frac{\lambda}{j-\lambda}\bigg|^2\right)^{1/2}
         \left( \sum^{\infty}_{|j|=N+1}|a_{j}|^2\right)^{1/2}\notag\\
    \leq&\frac{2\lambda}{\sqrt{N}L^{3/2}}\|\vec{J}\|_{\vec{p},0}.\notag
%    \leq & C_4\delta+\frac{C_5}{C_N}\lambda N^2\delta+\frac{2\lambda}{\sqrt{N}}\|\vec{J}\|_0\notag
\end{align}
The proof is completed.
\end{proof}

%\begin{remark}
%From the proof of Theorem \ref{th:ab}, we see that the formulas $\eqref{eq:al_form}$ and $\eqref{eq:bl_form}$ also hold for $\Gamma_r=\{\vec{x}\in \mathbb{R}^3:|\vec{x}|=r>R\}$ in place of $\Gamma_R$, and the approximate formula $\eqref{eq:a0N}$ can also be established on $\Gamma_r$ in place of $\Gamma_R$.
%\end{remark}

%--------------------------------------------------------------------------------------------------------------
%---------------------------------------------------------------------------------------------------------------

\section{Stability analysis}
Let $\vec{J}_N^{\delta}$ be the Fourier approximation in case that the measured data is perturbed with the noise level $\delta$. In this section, we will establish the calculation formula $\vec{J}_N^{\delta}$ and give an estimate of the error between $\vec{J}$ and $\vec{J}_N^{\delta}$.

In general, only data with noise $\hat{\vec{x}}\times \vec{E}^{\delta}\big|_{\Gamma_R}\in L_t^2({\Gamma_R})$ can be measured, and the data $\hat{\vec{x}}\times \nabla\times\vec{E}^{\delta}\big|_{\Gamma_R}\in L_t^2({\Gamma_R})$ cannot be computed from $\hat{\vec{x}}\times \vec{E}^{\delta}\big|_{\Gamma_R}\in L_t^2({\Gamma_R})$ by using the electric-to-magnetic Calder\'{o}n operator \eqref{eq:Calderon}. In this paper, we follow \cite{four} to overcome this difficulty. For simplicity,  we write $\vec{E}(\vec{x}; k)$ for $\vec{E}(\vec{x}; k, \vec{p})$ with a fixed polarization $\vec{p}\in\mathbb{P}$.

Using spherical coordinates $(r,\theta,\varphi)$, we can write the measured noisy data on $\Gamma_R$ as
\begin{align*}
   \hat{\vec{x}}\times \vec{E}^{\delta}(\vec{x}; k)\big|_{\Gamma_R}=
   \sum_{n=1}^{\infty}\sum_{m=-n}^{n}\left(\alpha_{k,n,m}^{\delta}\vec{U}_n^m
               +\beta_{k,n,m}^{\delta}\vec{V}_n^m\right),
\end{align*}
where
\begin{align}
    \alpha_{k,n,m}^{\delta}=\int_{\mathbb{S}^2} \hat{\vec{x}}\times \vec{E}^{\delta}(\vec{x}; k) \cdot \overline{\vec{U}_{n}^{m}} \dif s(\vec{x}),\label{eq:cal1}\\
     \beta_{k,n,m}^{\delta}=\int_{\mathbb{S}^2} \hat{\vec{x}}\times \vec{E}^{\delta}(\vec{x}; k) \cdot \overline{\vec{V}_n^m}\,\dif s(\vec{x}).\label{eq:cal2}
\end{align}
Then, for $r>R$, we define
\begin{align}
\hat{\vec{x}}\times \vec{E}(\vec{x};k)=\sum_{n=1}^{\infty}\sum_{m=-n}^{n}
\left(\frac{h_{n}^{(1)}(kr)}{h_{n}^{(1)}(kR)}\alpha_{k,n,m}^{\delta}\vec{U}_{n}^{m}+
\frac{R}{r}\frac{z_{n}^{(1)}(kr)}{z_{n}^{(1)}(kR)}\beta_{k,n,m}^{\delta}\vec{V}_{n}^{m}\right)
.\notag
\end{align}
\par
Take $r=\rho>R$, and let $\Gamma_\rho:=\{\vec{x}\in \mathbb{R}^3|\ |\vec{x}|=\rho\}$, we have
\begin{align}
&\hat{\vec{x}}\times \vec{E}^{\delta}\big|_{\Gamma_\rho}=\sum_{n=1}^{\infty}\sum_{m=-n}^{n}
        \left(\frac{h_{n}^{(1)}(k\rho)}{h_{n}^{(1)}(kR)}\alpha_{k,n,m}^{\delta}
        \vec{U}_{n}^{m}+\frac{R}{\rho}\frac{z_{n}^{(1)}(k\rho)}{z_{n}^{(1)}(kR)}
        \beta_{k,n,m}^{\delta}\vec{V}_{n}^{m}\right),\label{eq:cal3}\\
&\hat{\vec{x}}\times\nabla\times\vec{E}^{\delta}\big|_{\Gamma_\rho}=\sum_{n=1}^{\infty}\sum_{m=-n}^{n}
        \left(k^2R\frac{h_{n}^{(1)}(k\rho)}{z_{n}^{(1)}(kR)}\beta_{k,n,m}^{\delta} \vec{U}_{n}^{m}
        +\frac{1}{\rho}\frac{z_{n}^{(1)}(k\rho)}{h_{n}^{(1)}(kR)}
        \alpha_{k,n,m}^{\delta}\vec{V}_{n}^{m}\right).\label{eq:cal4}
\end{align}
\par
Now, we propose modified formulas for the approximate Fourier coefficients
\begin{align}
&a_{\vec{l}}^{\delta}=\frac{1}{\mi k|\vec{v}|^2 L^3}
     \int_{\Gamma_{\rho}} \left( \hat{\vec{x}}\times \nabla \times \vec{E}^{\delta} \cdot (\vec{w}\overline{\phi_{\vec{l}}(\vec{x})})
    +\hat{\vec{x}}\times  \vec{E}^{\delta} \cdot \nabla\times
    (\vec{w}\overline{\phi_{\vec{l}}(\vec{x})})\right )\dif s(\vec{x}), \label{eq:al_form_new}\\
&b_{\vec{l}}^{\delta}=-\frac{1}{2\pi k L^2 |\vec{v}|^2}
     \int_{\Gamma_{\rho}} \left( \hat{\vec{x}}\times \nabla \times \vec{E}^{\delta} \cdot (\vec{v}\overline{\phi_{\vec{l}}(\vec{x})})
    +\hat{\vec{x}}\times  \vec{E}^{\delta} \cdot \nabla\times
    (\vec{v}\overline{\phi_{\vec{l}}(\vec{x})})\right )\dif s(\vec{x}), \label{eq:bl_form_new}
\end{align}
for $\vec{l}\in \mathbb{Z}^3,\ 1\leq |\vec{l}|\leq N$, where $k=2\pi |\vec{l}|/L$, and
\begin{align}
    a_{\vec{0}}^{N,\delta}= &\frac{\lambda\pi}{\sin{\lambda\pi}} \bigg(\frac{1}{\mi k^*L^3|\vec{p}\times\vec{l}^*|^2}
    \int_{\Gamma_{\rho}} \Big( \hat{\vec{x}}\times \nabla \times \vec{E}^{\delta} \cdot (\vec{w}^*\overline{\phi_{\vec{l}^*}(\vec{x})})\notag \\
    &+\hat{\vec{x}}\times  \vec{E}^{\delta} \cdot \nabla\times(\vec{w}^*\overline{\phi_{\vec{l}^*}(\vec{x})}) \Big )\dif s(\vec{x})
    -\sum^{N}_{j=1}\frac{\sin(j-\lambda)\pi}{(j-\lambda)\pi}a_j^{\delta}\bigg),
    \label{eq:a0_form_new}
\end{align}
where $a_{j}^{\delta}=a_{\vec{l}}^{\delta}$ for $\vec{l}=(j,0,0)^{\top}$.
\par
By using $\eqref{eq:al_form_new},\eqref{eq:bl_form_new}$ and $\eqref{eq:a0_form_new}$, we define the function $\vec{J}_N^{\delta}$ by
\begin{align}
    \vec{J}_N^{\delta}(\vec{x}):=a_{\vec{0}}^{N,\delta}\vec{p}+
    \sum_{1\leq|\vec{l}|\leq N} \left(a_{\vec{l}}^{\delta}\vec{p}+\frac{2\pi\mi}{L} b_{\vec{l}}^{\delta}(\vec{p}\times \vec{l})\right)\phi_{\vec{l}}(\vec{x}),\label{eq:Jnd}
\end{align}
which will be an approximation of the source $\vec{J}$.

Let the measured noisy data $\hat{\vec{x}}\times \vec{E}^{\delta}(\cdot; k)\big|_{\Gamma_R}\in L_t^2({\Gamma_R})$ satisfy
\begin{align}
\|\hat{\vec{x}}\times\vec{E}(\vec{x};k)-\hat{\vec{x}}\times\vec{E}^{\delta}(\vec{x};k)\|_{L_t^2(\Gamma_R)}
\leq \delta \| \hat{\vec{x}}\times \vec{E}(\vec{x};k)\|_{L_t^2(\Gamma_R)},\label{eq:delerr}
\end{align}
where $0<\delta<1$. First we give an estimate on
$\| \hat{\vec{x}}\times \vec{E}(\vec{x};k)\|_{L_t^2(\Gamma_R)}$.
\newtheorem{Lemma}{Lemma}[section]
\begin{Lemma}\label{lm:kM}
Let $k\geq k^*$, then there exists a positive constant $M$ such that
\begin{align}
    \|\hat{\vec{x}}\times \vec{E}(\vec{x};k)\|_{L_t^2({\Gamma_R})}\leq kM\|\vec{J}\|_{\vec{p},0},\label{eq:J_norm}
\end{align}
where $M$ depends only on $\lambda,\ L$ and $R$.
\end{Lemma}
\begin{proof}
Introduce the fundamental solution of the Helmholtz equation
\[
    \Phi_k(\vec{x},\vec{y})=\frac{\exp(\mi k|\vec{x}-\vec{y}|)}{4\pi|\vec{x}-\vec{y}|},
    \quad \vec{x}\neq \vec{y},\quad\vec{x},\vec{y}\in \mathbb{R}^3.
\]
Then the unique solution to equation $\eqref{eq:maxwell}$ with the radiation condition $\eqref{eq:radiation}$ can be written as (see, e.g., \cite{Neld})
\begin{align}
\vec{E}(\vec{x};k)=&\mi k\int_{V_0}\left(\vec{I}+\frac{\nabla\nabla^{\top}}{k^2}\right)
                    \Phi_k(\vec{x},\vec{y})\vec{J}(\vec{y})\,\dif \vec{y}\notag\\
                  =&\mi k\int_{V_0}\Phi_k(\vec{x},\vec{y})
                    \bigg(\left(1+\frac{\mi}{k}-\frac{1}{k^2|\vec{x}-\vec{y}|^2}\right)\vec{I} \notag\\
                    &+\left(\frac{3}{k^2|\vec{x}-\vec{y}|^2}- \frac{3\mi}{k|\vec{x}-\vec{y}|}-1\right)
            \frac{(\vec{x}-\vec{y})(\vec{x}-\vec{y})^{\top}}{|\vec{x}-\vec{y}|^2}\bigg)
                       \vec{J}(\vec{y})\,\dif \vec{y},\label{eq:M_solve}
\end{align}
where $\vec{x}\in \mathbb{R}^3\backslash V_0$ and $\vec{I}$ denotes the $3\times 3$ identity matrix.

Let
\[
 \tau_0:=\min\{|\vec{x}-\vec{y}|\mid \vec{x}\in\Gamma_R,\ \vec{y}\in V_0 \}=R-\frac{\sqrt{3}}{2}L,
\]
then for all $\vec{x}\in\Gamma_R$ and $\vec{y}\in V_0$, we obtain
    \begin{align}
    &\left|1+\frac{\mi}{k}-\frac{1}{k^2|\vec{x}-\vec{y}|^2}\right|\leq 1+\frac{1}{k}+\frac{1}{k^2\tau_0^2},\notag\\
    &\left|\frac{3}{k^2|\vec{x}-\vec{y}|^2}- \frac{3\mi}{k|\vec{x}-\vec{y}|}-1\right|\leq 1+\frac{3}{k^2\tau_0^2},\notag
    \end{align}
and
    \begin{align}
    \left|\Phi_k(\vec{x},\vec{y})\right| \leq\frac{1}{4\pi\tau_0}.\notag
    \end{align}
Then, we see
\[
|\hat{\vec{x}}\times\vec{E}(\vec{x};k)|\leq \frac{k}{4\pi\tau_0}\left(2+\frac{1}{k}
            +\frac{4}{k^2\tau_0^2}\right)L^{3/2}\|\vec{J}\|_{\vec{p},0},\quad \vec{x}\in\Gamma_R.
\]
Furthermore, we obtain
\[
    \|\hat{\vec{x}}\times \vec{E}(\vec{x};k)\|_{L_t^2({\Gamma_R})}\leq \frac{kR}{2\sqrt{\pi}\tau_0}\left(2+\frac{1}{k}+\frac{4}{k^2\tau_0^2}\right)
    L^{3/2}\|\vec{J}\|_{\vec{p},0}.
\]
Hence the estimate $\eqref{eq:J_norm}$ follows by noticing $k\geq k^*$ and taking
\begin{equation}\label{eq:conM}
M:=\frac{R}{2\sqrt{\pi}\tau_0}\left(2+\frac{1}{k^*}+\frac{4}{k^{*2}\tau_0^2}\right)L^{3/2}.
\end{equation}
\end{proof}

Now we recall and establish some estimates of the spherical Hankel functions $h_n^{(1)}(t)$ and $z_n^{(1)}(t)=h_n^{(1)}(t)+th_n^{(1)\prime}(t)$, which will play an important role in the stability analysis. To this end, we need the following estimates which can be found in \cite{four, Chen}.

\begin{Lemma}\label{lm:delta}
Let $n\in \mathbb{N}$ and $t>0$. Then it holds that
\[
\left|\frac{z_n^{(1)}(t)}{h_n^{(1)}(t)}\right|\geq \frac{n(n+1)}{2t^2+n+1}.
\]
\end{Lemma}

\begin{Lemma}\label{lm:dhh}
Let $n\in\mathbb{N}$ and $0<R<\rho$. Then the following estimates hold
\begin{equation}\label{eq:dhh1}
    \left|\frac{h_n^{(1)\prime}(k\rho)}{h_n^{(1)}(kR)}\right|\leq
       \begin{cases}
                c_1,   & k\geq \dfrac{2\pi}{L}, \\
                \dfrac{c_2}{k},  & 0<k\leq\dfrac{1}{2},
    \end{cases}
\end{equation}
where
\begin{align*}
    &c_1=\max\left\{ 9,\ 1+\frac {c_4L\me^{-1+c_3(1-R/\rho)}}  {2\pi c_3(\rho-R)} \right\},\\
    &c_2=\max\left\{\frac{9}{2},\ \frac{1}{2}+\frac{c_4}{c_3(\rho-R)\sqrt{2c_3 \me} } \right\}, \\
    &c_3=\frac{\sqrt{15}}{4},\notag\\%=0.9682458
    &c_4=\frac{25+11\me^{11/25}}{25-11\me^{11/25}}\left(\frac{16}{15}\right)^{1/4}.%=5.39937
\end{align*}
\end{Lemma}

From Lemma \ref{lm:delta} and Lemma \ref{lm:dhh}, we can derive the following result.

\begin{Lemma}\label{lm:hklest}
Let $n\in\mathbb{N}$ and $0<R<\rho$. Then the following estimates hold for $k>0$
\begin{align}
    &\left|\frac{h_n^{(1)}(k\rho)}{h_n^{(1)}(kR)}\right|\leq 1,  \label{eq:hklest1}\\
    &\left|\frac{z_n^{(1)}(k\rho)}{z_n^{(1)}(kR)}\right|\leq \frac{(4+10kR)\rho}{R},\label{eq:hklest2}\\
    &\left|\frac{h_n^{(1)}(k\rho)}{z_n^{(1)}(kR)}\right|\leq 7, \label{eq:hklest3}\\
    &\left|\frac{z_n^{(1)}(k\rho)}{h_n^{(1)}(kR)}\right|\leq
    \begin{cases}
    C_1 k\rho, & k\geq \dfrac{2\pi}{L}, \\
    C_2 \rho,  & 0<k\leq \dfrac{1}{2},
    \end{cases}\label{eq:hklest4}
\end{align}
where constants $C_1$ and $C_2$ depend only on $L,\ R$ and $\rho$.
\end{Lemma}

\begin{proof}
We get $\eqref{eq:hklest1}$ from Lemma 3.3 in \cite{four}, and it still holds for $n=0$.

We now turn to prove $\eqref{eq:hklest2}$. Since $h_n^{(1)\prime}(t)=-(n+1)h_n^{(1)}(t)/t+h_{n-1}^{(1)}(t)$, we have
\begin{align*}
\left|\frac{R}{\rho}\frac{z_n^{(1)}(k\rho)}{z_n^{(1)}(kR)}\right|
        &=\left|\frac{R}{\rho}\frac{h_n^{(1)}(k\rho)+k\rho h_n^{(1)\prime}(k\rho)}{h_n^{(1)}(kR)}\right|\left|\frac{h_n^{(1)}(kR)}{z_n^{(1)}(kR)}\right|\\
        &=\left|\frac{R}{\rho}\frac{-nh_n^{(1)}(k\rho)+k\rho h_{n-1}^{(1)}(k\rho)}{h_n^{(1)}(kR)}\right|\left|\frac{h_n^{(1)}(kR)}{z_n^{(1)}(kR)}\right|\\
        &\leq \left( n\left|\frac{h_n^{(1)}(k\rho)}{h_n^{(1)}(kR)}\right|
         +kR\left|\frac{h_{n-1}^{(1)}(k\rho)}{h_{n-1}^{(1)}(kR)}
         \frac{h_{n-1}^{(1)}(kR)}{h_{n}^{(1)}(kR)}\right|\right)
         \left|\frac{h_n^{(1)}(kR)}{z_n^{(1)}(kR)}\right|\\
        &\leq \frac{n+kR}{|\gamma_n(kR)|},
\end{align*}
where we use the inequality that $|h_{n-1}^{(1)}(t)|\leq |h_{n}^{(1)}(t)|$ for $n\geq 1, t>0$ \cite[p.8]{four} and $\gamma_n:=z_n^{(1)}/h_n^{(1)}$.

We now estimate $|\gamma_n(t)|$ for $t>0$. Since ${h_n^{(1)}}'(t)=nh_n^{(1)}(t)/t-h_{n+1}^{(1)}(t)$, we have
\begin{equation}\label{eq:delta}
  \gamma_n(t)=\frac{z_n^{(1)}(t)}{h_n^{(1)}(t)}=n+1-t \frac{h_{n+1}^{(1)}(t)}{h_n^{(1)}(t)},
\end{equation}
which implies, for $kR\geq 5n/4+1, |\gamma_n(kR)|\geq kR-(n+1)\geq n/4$ and thus
\[
 \frac{n+kR}{|\gamma_n(kR)|}\leq 4\left(1+\frac{1}{n}kR\right).
\]
For $kR\leq 5n/4+1$, we resort to the estimate in Lemma \ref{lm:delta} to get
    \begin{align}
    \frac{1}{|\gamma_n(kR)|}\leq \frac{2kR(\frac{5}{4}n+1)+n+1}{n(n+1)}
     \leq\frac{1}{n}\left(1+\frac{5}{2}kR\right),\notag
    \end{align}
hence
    \begin{align}
    \frac{n+kR}{|\gamma_n(kR)|}\leq  \frac{n+(\frac{5}{4}n+1)}{n}\left(1+\frac{5}{2}kR\right)
    \leq  4\left(1+\frac{5}{2}kR\right),\notag
    \end{align}
and $\eqref{eq:hklest2}$ follows.
\par
To show $\eqref{eq:hklest3}$, we have
\begin{align}
    \left|\frac{h_n^{(1)}(k\rho)}{z_n^{(1)}(kR)}\right|
    =\left|\frac{h_n^{(1)}(k\rho)}{h_n^{(1)}(kR)}\right|  \left|\frac{h_n^{(1)}(kR)}{z_n^{(1)}(kR)}\right|
    \leq\frac{1}{|\gamma_n(kR)|}, \notag
\end{align}
then for $kR\geq n+5/4$, in light of $\eqref{eq:delta}$ we get $|\gamma_n(kR)|\geq kR-(n+1)\geq 1/4$ and thus $|\gamma_n(kR)|^{-1}\leq 4$, for $kR\leq n+5/4$, using the estimate in Lemma \ref{lm:delta}, we get $|\gamma_n(kR)|^{-1}\leq 7$, and we have thus proved $\eqref{eq:hklest3}$ .
\par
To prove $\eqref{eq:hklest4}$, we have
    \begin{align}
    \left|\frac{z_n^{(1)}(k\rho)}{h_n^{(1)}(kR)}\right|
    &\leq \left|\frac{h_n^{(1)}(k\rho)}{h_n^{(1)}(kR)}\right|+k\rho\left|\frac{h_n^{(1)\prime}(k\rho)}{h_n^{(1)}(kR)}\right|,\notag
    \end{align}
then, using $\eqref{eq:dhh1}$, we obtain
\begin{align*}
& \left|\frac{z_n^{(1)}(k\rho)}{h_n^{(1)}(kR)}\right|\leq 1+c_1k\rho\leq \left(\frac{L}{2\pi\rho}+c_1\right)k\rho=C_1k\rho,\quad k\geq \frac{2\pi}{L},\\
& \left|\frac{z_n^{(1)}(k\rho)}{h_n^{(1)}(kR)}\right|\leq 1+c_2\rho= \left(\frac{1}{\rho}+c_2\right)\rho=C_2\rho,\quad 0<k\leq \frac{1}{2},
\end{align*}
where constants $C_1$ and $C_2$ depend only on $L,\ R$ and $\rho$.
\end{proof}
From Lemma \ref{lm:kM} and \ref{lm:hklest}, we derive the following result.
\begin{theorem}\label{th:Eerr}
Let $0<R<\rho$ and $\lambda<L/(4\pi)$. Then there exist constants $C_1$ and $C_2$ depending only on $L,\ R,\ \lambda$ and
$\rho$ such that
\begin{align}
    &\| \hat{\vec{x}}\times \vec{E}(\vec{x};k)-\hat{\vec{x}}\times\vec{E}^{\delta}(\vec{x};k)\|_{L_t^2(\Gamma_{\rho})}
    \leq (4+10kR)kM\|\vec{J}\|_{\vec{p},0}\delta, \label{eq:err_dirE}\\
    &\| \hat{\vec{x}}\times \nabla\times\vec{E}(\vec{x};k)-\hat{\vec{x}}\times \nabla\times\vec{E}^{\delta}(\vec{x};k)\|_{L_t^2(\Gamma_{\rho})}\leq (C_1+7Rk)k^2M\|\vec{J}\|_{\vec{p},0}\delta, \label{eq:err_neuE1}\\
    &\| \hat{\vec{x}}\times \nabla\times\vec{E}(\vec{x};k^*)-\hat{\vec{x}}\times \nabla\times\vec{E}^{\delta}(\vec{x}; k^*)\|_{L_t^2(\Gamma_{\rho})}\leq (C_2+7Rk^{*2})k^*M\|\vec{J}\|_{\vec{p},0}\delta,\label{eq:err_neuE2}
\end{align}
where the constant $M$ is defined by \eqref{eq:conM}.
\end{theorem}

\begin{proof}
From \eqref{eq:delerr}, \eqref{eq:J_norm}, \eqref{eq:hklest1} and \eqref{eq:hklest2}, we have
\begin{align*}
    &\| \hat{\vec{x}}\times \vec{E}(\vec{x};k)-\hat{\vec{x}}\times\vec{E}^{\delta}(\vec{x};k)\|_{L_t^2(\Gamma_{\rho})}\\
    \leq& \left(\sum_{n=1}^{\infty}\sum_{m=-n}^{n}\Big( |\alpha_{k,n,m}-\alpha_{k,n,m}^{\delta}|^2 +(4+10kR)^2|\beta_{k,n,m}-\beta_{k,n,m}^{\delta}|^2 \Big)\right)^{1/2}\\
    \leq& (4+10kR)\| \hat{\vec{x}}\times \vec{E}(\vec{x};k)-\hat{\vec{x}}\times\vec{E}^{\delta}(\vec{x};k)\|_{L_t^2(\Gamma_{R})}\\
    \leq& (4+10kR)kM\|\vec{J}\|_{\vec{p},0}\delta,
\end{align*}
which gives \eqref{eq:err_dirE}.

Using \eqref{eq:delerr},\ \eqref{eq:J_norm},\ \eqref{eq:hklest3} and \eqref{eq:hklest4}, we have
\begin{align*}
    &\| \hat{\vec{x}}\times \nabla\times\vec{E}(\vec{x};k)-\hat{\vec{x}}\times
     \nabla\times\vec{E}^{\delta}(\vec{x};k)\|_{L_t^2(\Gamma_{\rho})}\\
    \leq&\left(\sum_{n=1}^{\infty}\sum_{m=-n}^{n}\Big(C_1^2k^2|\alpha_{k,n,m}-\alpha_{k,n,m}^{\delta}|^2
     +49k^4R^2|\beta_{k,n,m}-\beta_{k,n,m}^{\delta}|^2 \Big)\right)^{1/2}\\
    \leq& (C_1+7Rk)k\| \hat{\vec{x}}\times \vec{E}(k,x)-\hat{\vec{x}}\times\vec{E}^{\delta}(k,x)\|_{L_t^2(\Gamma_{R})}\\
    \leq& (C_1+7Rk)k^2M\|\vec{J}\|_{\vec{p},0}\delta,
\end{align*}
which yields the desired estimate $\eqref{eq:err_neuE1}$.

Similar to \eqref{eq:err_neuE1}, using \eqref{eq:delerr}, \eqref{eq:J_norm}, \eqref{eq:hklest3} and \eqref{eq:hklest4}, and noticing $k^*=2\pi\lambda/L< 1/2$, we have
\begin{align}
    &\| \hat{\vec{x}}\times \nabla\times\vec{E}(k^*,\vec{x})-\hat{\vec{x}}\times \nabla\times\vec{E}^{\delta}(k^*,\vec{x})\|_{L_t^2(\Gamma_{\rho})}\notag\\
    \leq& \left(\sum_{n=1}^{\infty}\sum_{m=-n}^{n}\Big( C_2^2|\alpha_{k^*,n,m}-\alpha_{k^*,n,m}^{\delta}|^2 +49k^4R^2|\beta_{k^*,n,m}-\beta_{k^*,n,m}^{\delta}|^2 \Big)\right)^{1/2}\notag\\
    \leq& (C_2+7Rk^{*2})\| \hat{\vec{x}}\times \vec{E}(k^*,\vec{x})-\hat{\vec{x}}\times\vec{E}^{\delta}(k^*,\vec{x})\|_{L_t^2(\Gamma_{R})}\notag\\
    \leq& (C_2+7Rk^{*2})k^*M\|\vec{J}\|_{\vec{p},0}\delta, \notag
\end{align}
then we obtain $\eqref{eq:err_neuE2}$, and the proof is completed.
\end{proof}

Thanks to Theorem \ref{th:Eerr}, we can show the error estimates of $|a_{\vec{l}}-a_{\vec{l}}^{\delta}|$ and $|b_{\vec{l}}-b_{\vec{l}}^{\delta}|$.

\begin{theorem}\label{th:aberr}
For $\vec{l}\in\mathbb{Z}^3,\ |\vec{l}|\leq N$, $k=2\pi |\vec{l}|/L$ and $\lambda<1/2$, we have
\begin{align}
    & |a_{\vec{l}}-a_{\vec{l}}^{\delta}|\leq \frac{C_3}{C_N}k^2\|\vec{J}\|_{\vec{p},0}\delta, \quad 1\leq|\vec{l}|\leq N,\label{eq:aerr}\\
    & |b_{\vec{l}}-b_{\vec{l}}^{\delta}|\leq \frac{C_3}{C_N}k\|\vec{J}\|_{\vec{p},0}\delta, \quad 1\leq|\vec{l}|\leq N,\label{eq:berr}\\
    & |a_{\vec{0}}-a_{\vec{0}}^{\delta}|\leq \left(C_4\delta+\frac{4\pi^2C_3\lambda\delta}{C_N L^2}(N^2+3N)+\frac{2\lambda}{\sqrt{N}L^{3/2}}\right)\|\vec{J}\|_{\vec{p},0}, \label{eq:a0err}
\end{align}
where the constant $C_N$ depends only on $N$, and $C_3$ and $C_4$ depend only on $L$, $R$, $\rho$, $\lambda$ and $\vec{J}$.
\end{theorem}
\begin{proof}
For $1\leq|\vec{l}|\leq N$, by using $\eqref{eq:err_dirE},\ \eqref{eq:err_neuE1}$, we have
\begin{align}
    &|a_{\vec{l}}-a_{\vec{l}}^{\delta}|\notag\\
    \leq& \frac{1}{k|\vec{v}_{\vec{l}}|^2 L^3}
    \int_{\Gamma_{\rho}} \Big( |(\hat{\vec{x}}\times \nabla \times
    \vec{E}(\vec{x};k)-\hat{\vec{x}}\times \nabla \times \vec{E}^{\delta}(\vec{x};k)) \cdot
     (\vec{w}_{\vec{l}}\overline{\phi_{\vec{l}}(\vec{x})})|\notag\\
    &+|(\hat{\vec{x}}\times  \vec{E}(\vec{x};k)-\hat{\vec{x}}\times  \vec{E}^{\delta}(\vec{x};k)) \cdot
     \nabla\times(\vec{w}_{\vec{l}}\overline{\phi_{\vec{l}}(\vec{x})})| \Big)\dif s(\vec{x}) \notag\\
    \leq&  \frac{ 2\sqrt{\pi}\rho}{k|\vec{v}_{\vec{l}}|^2 L^3} \Big(\| \hat{\vec{x}}\times \nabla\times\vec{E}(\vec{x};k)-\hat{\vec{x}}\times \nabla\times\vec{E}^{\delta}(\vec{x};k)\|_{L_t^2(\Gamma_{\rho})}|\vec{l}| |\vec{v}_{\vec{l}}|  \notag\\
    &+ \| \hat{\vec{x}}\times \vec{E}(\vec{x};k)-\hat{\vec{x}}\times\vec{E}^{\delta}(\vec{x};k)\|_{L_t^2(\Gamma_{\rho})}
    |\vec{l}|^2 |\vec{v}_{\vec{l}}|\frac{2\pi}{L}\Big)\notag\\
    \leq& \frac{ 2\sqrt{\pi}\rho}{k|\vec{v}_{\vec{l}}| L^3}\left((C_1+7Rk)M\|\vec{J}\|_{\vec{p},0}k^2|\vec{l}|\delta
    +(4+10kR)M\|\vec{J}\|_{\vec{p},0}k\frac{2\pi}{L}|\vec{l}|^2\delta\right)\notag\\
    \leq& \frac{2\sqrt{\pi}\rho}{|\vec{p}\times\hat{\vec{l}}|L^3}(C_1+4+17kR)
        M\|\vec{J}\|_{\vec{p},0}k\delta, \notag
\end{align}
and
    \begin{align}
    &|b_{\vec{l}}-b_{\vec{l}}^{\delta}|\notag\\
    \leq& \frac{1}{2\pi k|\vec{v}_{\vec{l}}|^2 L^2}
    \int_{\Gamma_{\rho}} \Big( |(\hat{\vec{x}}\times \nabla \times \vec{E}(\vec{x};k)-\hat{\vec{x}}\times \nabla \times \vec{E}^{\delta}(\vec{x};k)) \cdot (\vec{v}_{\vec{l}}\overline{\phi_{\vec{l}}(\vec{x})})|\notag\\
    &+|(\hat{\vec{x}}\times  \vec{E}(\vec{x};k)-\hat{\vec{x}}\times  \vec{E}^{\delta}(\vec{x};k)) \cdot \nabla\times(\vec{v}_{\vec{l}}\overline{\phi_{\vec{l}}(\vec{x})})| \Big)\dif s(\vec{x}) \notag\\
    \leq&  \frac{ 2\sqrt{\pi}\rho}{2\pi k|\vec{v}_{\vec{l}}|^2 L^2} \Big(\| \hat{\vec{x}}\times \nabla\times\vec{E}(\vec{x};k)-\hat{\vec{x}}\times \nabla\times\vec{E}^{\delta}(\vec{x};k)\|_{L_t^2(\Gamma_{\rho})} |\vec{v}_{\vec{l}}|  \notag\\
    &+ \| \hat{\vec{x}}\times \vec{E}(\vec{x};k)-\hat{\vec{x}}\times\vec{E}^{\delta}(\vec{x};k)\|_{L_t^2(\Gamma_{\rho})}
    |\vec{l}| |\vec{v}_{\vec{l}}|\frac{2\pi}{L}\Big)\notag\\
    \leq& \frac{ 2\sqrt{\pi}\rho}{2\pi k|\vec{v}_{\vec{l}}| L^2}\left((C_1+7Rk)M\|\vec{J}\|_{\vec{p},0}k^2\delta
    +(4+10kR)M\|\vec{J}\|_{\vec{p},0}k\frac{2\pi}{L}|\vec{l}|\delta\right)\notag\\
    \leq& \frac{2\sqrt{\pi}\rho}{|\vec{p}\times\hat{\vec{l}}|L^3}(C_1+4+17kR)
       M\|\vec{J}\|_{\vec{p},0}\delta,\notag
    \end{align}
where $\hat{\vec{l}}$ denotes the unit vector $\vec{l}/|\vec{l}|$.

Since $2\pi\leq kL$, we get $C_1+4+17kR\leq ((C_1+4)L/2\pi+17R)k$. Let
\begin{align*}
    &C_3:=\frac{2\sqrt{\pi}\rho}{L^3}\left((C_1+4)\frac{L}{2\pi}+17R\right)M, \\
    &C_N:=\min_{1\leq|\vec{l}|\leq N}|\vec{p}\times\hat{\vec{l}}|,
\end{align*}
then we obtain
\[
    |a_{\vec{l}}-a_{\vec{l}}^{\delta}|\leq\frac{C_3}{C_N}k^2\|\vec{J}\|_{\vec{p},0}\delta,\quad
    |b_{\vec{l}}-b_{\vec{l}}^{\delta}|\leq\frac{C_3}{C_N}k\|\vec{J}\|_{\vec{p},0}\delta,
\]
thus we obtain \eqref{eq:aerr} and \eqref{eq:berr}.

To prove $\eqref{eq:a0err}$, from $\eqref{eq:a0}$ and $\eqref{eq:a0_form_new}$,
we obtain
    \begin{align}
         &|a_{\vec{0}}-a_{\vec{0}}^{\delta}| \notag\\
    \leq&\frac{\lambda\pi}{\sin{\lambda\pi}}\bigg(
    \bigg|\frac{1}{\mi k^*L^3|\vec{p}\times\vec{l}^*|^2}
    \int_{\Gamma_{\rho}} \Big( \hat{\vec{x}}\times \nabla \times (\vec{E}(\vec{x};k)-\vec{E}^{\delta}(\vec{x};k)) \notag\\
    &\cdot (\vec{w}^*\overline{\phi_{\vec{l}^*}(\vec{x})})
    +\hat{\vec{x}}\times  (\vec{E}(\vec{x};k)-\vec{E}^{\delta}(\vec{x};k)) \cdot \nabla\times(\vec{w}^*\overline{\phi_{\vec{l}^*}(\vec{x})}) \Big)\dif s(\vec{x})   \bigg |\notag\\
    &+\sum^{N}_{|j|=1}\bigg|  \frac{\sin{(j-\lambda)\pi}}{(j-\lambda)\pi}(a_{\vec{l}}-a_{\vec{l}}^{\delta})  \bigg |
    +\sum^{\infty}_{|j|=N+1}\bigg|\frac{\sin{(j-\lambda)\pi}}{(j-\lambda)\pi}a_{\vec{l}}\bigg |\bigg)\notag\\
    =&I_1+I_2+I_3.\notag
    \end{align}
From $\eqref{eq:err_dirE}$ and $\eqref{eq:err_neuE2}$, we obtain
    \begin{align}
    I_1%=
%    &\bigg|\frac{\lambda\pi}{\sin{(\lambda\pi)}} \frac{1}{\mi k^*L^3|\vec{p}\times\vec{l}^*|^2}
%    \int_{\Gamma_{\rho}} \bigg[ \hat{\vec{x}}\times \nabla \times (\vec{E}-\vec{E}^{\delta}) \cdot (\vec{w}^*\overline{\phi_{\vec{l}^*}(\vec{x})})\notag\\
%    &+\hat{\vec{x}}\times  (\vec{E}-\vec{E}^{\delta}) \cdot \nabla\times(\vec{w}^*\overline{\phi_{\vec{l}^*}(\vec{x})}) \bigg]\dif A   \bigg |\notag\\
    \leq&\frac{\lambda\pi}{k^*L^3|\vec{p}\times\vec{l}^*|^2\sin{\lambda\pi}}
    \Big(  2\sqrt{\pi}\rho(C_2+7k^{*2}R)k^*M\|\vec{J}\|_{\vec{p},0}\delta  |\vec{l}^*|   |\vec{p}\times\vec{l}^*|  \notag\\
    &+2\sqrt{\pi}\rho(4+10k^*R) k^*M\|\vec{J}\|_{\vec{p},0}\delta\frac{2\pi}{L}|\vec{l}^*|^2 |\vec{p}\times\vec{l}^*| \Big) \notag\\
    \leq&\frac{2\lambda\pi^{3/2}\rho M}{L^3|\vec{p}\times\hat{\vec{l}}^*|\sin{\lambda\pi}}
        ((C_2+7k^{*2}R)+(4+10k^*R)k^*)\|\vec{J}\|_{\vec{p},0}\delta\notag\\
    \leq& C_4 \|\vec{J}\|_{\vec{p},0}\delta,\notag
    \end{align}
where $\hat{\vec{l}}^*=(1,0,0)^{\top}$ and
\begin{align}
C_4:=\frac{2\lambda\pi^{3/2}\rho M}{L^3|\vec{p}\times\hat{\vec{l}}^*|\sin{\lambda\pi}}
    (C_2+4k^*+17k^{*2}R).\notag
\end{align}
Using \eqref{eq:aerr}, and noticing $\lambda < 1/2$, we have
\begin{align*}
    I_2=& \frac{\lambda\pi}{\sin{\lambda\pi}}\sum^{N}_{|j|=1}\bigg|\frac{\sin (j-\lambda)\pi}{(j-\lambda)\pi}(a_{\vec{l}}-a_{\vec{l}}^{\delta}) \bigg |\\
    \leq& \frac{4\pi^2 C_3\|\vec{J}\|_{\vec{p},0}\lambda\delta}{C_N L^2}
            \sum^{N}_{|j|=1}\left|\frac{j^2}{j-\lambda}\right|
    = \frac{4\pi^2 C_3\|\vec{J}\|_{\vec{p},0}\lambda\delta}
       {C_N L^2}\sum^N_{j=1}\frac{2j^3}{j^2-\lambda^2}\\
    \leq& \frac{4\pi^2 C_3\|\vec{J}\|_{\vec{p},0}\lambda\delta}{C_N L^2}
          \sum^N_{j=1}\int_j^{j+1}\frac{2t^3}{t^2-\lambda^2}\,\dif t
    =\frac{4\pi^2 C_3\|\vec{J}\|_{\vec{p},0}\lambda\delta}
        {C_N L^2}\int_1^{N+1}\frac{2t^3}{t^2-\lambda^2}\,\dif t\\
    =&\frac{4\pi^2 C_3\|\vec{J}\|_{\vec{p},0}\lambda\delta}{C_N L^2}
      (N^2+2N+\lambda^2\ln{((N+1)^2-\lambda^2)}-\lambda^2\ln{(1-\lambda^2)})\\
    \leq &\frac{4\pi^2 C_3\|\vec{J}\|_{\vec{p},0}\lambda\delta}
       {C_N L^2}\left(N^2+2N+2\lambda^2N+\frac{1}{2}\right) \\
    \leq& \frac{4\pi^2 C_3\|\vec{J}\|_{\vec{p},0}\lambda\delta}{C_N L^2}(N^2+3N).
\end{align*}
Combining with $(\ref{eq:a0a0N})$
    \begin{align}
    I_3=|a_{\vec{0}}-a_{\vec{0}}^{N}|  \leq \frac{2\lambda}{\sqrt{N}L^{3/2}}\|\vec{J}\|_{\vec{p},0},\notag
    \end{align}
%        \begin{align}
%    I_3=&\sum^{\infty}_{|j|=N+1}\bigg|  \frac{\lambda\pi}{\sin{(\lambda\pi)}}\frac{\sin{((j-\lambda)\pi)}}{(j-\lambda)\pi}a_{\vec{l}}  \bigg |\notag\\
%    =&\sum^{\infty}_{|j|=N+1}\bigg|\frac{\lambda}{j-\lambda}a_{\vec{l}}  \bigg |\notag\\
%    \leq&\left( \sum^{\infty}_{|j|=N+1}\bigg|\frac{\lambda}{j-\lambda}\bigg|^2\right)^{1/2}
%         \left( \sum^{\infty}_{|j|=N+1}|a_{\vec{l}}|^2\right)^{1/2}\notag\\
%    \leq&\frac{2\lambda}{\sqrt{N}}\|\vec{J}\|_0.\notag
%%    \leq & C_4\delta+\frac{C_5}{C_N}\lambda N^2\delta+\frac{2\lambda}{\sqrt{N}}\|\vec{J}\|_0\notag
%    \end{align}
In summary, we obtain
\begin{align}
|a_{\vec{0}}-a_{\vec{0}}^{\delta}|\leq \left(C_4\delta+\frac{4\pi^2C_3\lambda\delta}{C_N L^2}(N^2+3N)+\frac{2\lambda}{\sqrt{N}L^{3/2}}\right)\|\vec{J}\|_{\vec{p},0}.\notag
\end{align}
This completes the proof.
\end{proof}
\par
Now we are in a position to present the main result.
\begin{theorem}\label{th:mainerr}
Let $\vec{J}$ be a function in $(H^{\sigma}_{\vec{p}}(V_0))^3$ and $0\leq\mu<\sigma$, then the following estimate holds
\begin{align}
    \|\vec{J}-\vec{J}_N^{\delta}\|_{\vec{p},\mu}
    \leq\left(C_5\delta+\frac{2\lambda}{\sqrt{N}}+C_6N^{\mu+7/2}\delta\right)
       \|\vec{J}\|_{\vec{p},0}+N^{\mu-\sigma}\|\vec{J}\|_{\vec{p},\sigma}.\label{eq:mainerr}
\end{align}
If, in addition, we take $\tau\geq 1$ and $N=[\tau\delta^{\frac{-1}{\mu+4}}]+1$, then
\begin{align}
    \|\vec{J}-\vec{J}_N^{\delta}\|_{\vec{p},\mu}
       \leq&C_5\|\vec{J}\|_{\vec{p},0}\delta+ \frac{1}{\tau^{\sigma-\mu}}
       \|\vec{J}\|_{\vec{p},\sigma}\delta^{\frac{\sigma-\mu}{\mu+4}}\notag\\
    &+\left( \frac{2\lambda}{\sqrt{\tau}}+(2\tau)^{\mu+7/2}C_6\right)\|\vec{J}\|_{\vec{p},0}
    \delta^{\frac{1}{2(\mu+4)}},\notag%\label{eq:mainerr1}
\end{align}
where $[X]$ denotes the largest integer that is smaller than $X+1$.
    \end{theorem}
\begin{proof}
From \eqref{eq:Jnorm} and Theorem \ref{th:aberr}, it is readily seen that
\begin{align}
    &\|\vec{J}_N-\vec{J}_N^{\delta}\|_{\vec{p},\mu}\notag\\
    \leq&  \Bigg( L^3\sum_{|\vec{l}|=0}^{N}
      \left(1+|\vec{l}|^2\right)^{\mu}|a_{\vec{l}}-a_{\vec{l}}^{\delta}|^2\notag
    +4\pi^2L\sum_{|\vec{l}|=1}^{N}\left(1+|\vec{l}|^2\right)^{\mu}|\vec{p}\times \vec{l}|^2|b_{\vec{l}}-b_{\vec{l}}^{\delta}|^2  \Bigg)^{1/2} \notag\\
    \leq& \left(C_4L^{3/2}\delta+\frac{4\pi^2 C_3\lambda\delta}{C_N \sqrt{L}}(N^2+3N)
       +\frac{2\lambda }{\sqrt{N}}\right)\|\vec{J}\|_{\vec{p},0}\notag\\
    &+ \frac{4\sqrt{2}\pi^2 C_3\|\vec{J}\|_{\vec{p},0}\delta}{C_N\sqrt{L}}
     \left( \sum_{|\vec{l}|=1}^{N}(1+|\vec{l}|^2)^{\mu}|\vec{l}|^4 \right)^{1/2}\notag\\
    \leq&  \left(C_4L^{3/2}\delta+\frac{4\pi^2 C_3\lambda\delta}{C_N \sqrt{L}}(N^2+3N)
       +\frac{2\lambda}{\sqrt{N}}\right)\|\vec{J}\|_{\vec{p},0}\notag\\
       &+ \frac{4\sqrt{2}\pi^2 C_3\delta}{C_N\sqrt{L}}(1+N^2)^{\mu/2}N^2(2N+1)^{3/2}
       \|\vec{J}\|_{\vec{p},0}\notag\\
    \leq& \left(C_5\delta+\frac{2\lambda }{\sqrt{N}}
               +C_6N^{\mu+7/2}\delta\right)\|\vec{J}\|_{\vec{p},0}\notag
\end{align}
where
\[
    C_5:=C_4L^{3/2},\quad
%=\frac{2\lambda\pi^{3/2}\rho M}{L^{3/2}|\vec{p}\times\hat{\vec{l}}^*|\sin{\lambda\pi}}
%    (C_2+4k^*+17k^{*2}R)
    C_6:=\frac{4\pi^2 C_3(4\lambda+3\sqrt{6}2^{\mu/2})}{C_N \sqrt{L}}.
\]
Hence, from Theorem \ref{th:JJNerr}, we know
\begin{align}
    \|\vec{J}-\vec{J}_N^{\delta}\|_{\vec{p},\mu}
    \leq& \|\vec{J}_N-\vec{J}_N^{\delta}\|_{\vec{p},\mu}+ \|\vec{J}-\vec{J}_N\|_{\vec{p},\mu} \notag\\
    \leq&  \left(C_5\delta+\frac{2\lambda}{\sqrt{N}}+C_6N^{\mu+7/2}\delta\right)
       \|\vec{J}\|_{\vec{p},0}+N^{\mu-\sigma}\|\vec{J}\|_{\vec{p},\sigma},\notag
\end{align}
which leads to $\eqref{eq:mainerr}$.
Further, let $\tau\geq 1$ and $N=[\tau\delta^{\frac{-1}{\mu+4}}]+1$ and we have
\begin{align}
    \|\vec{J}-\vec{J}_N^{\delta}\|_{\vec{p},\mu}
       \leq&C_5\|\vec{J}\|_{\vec{p},0}\delta+ \frac{1}{\tau^{\sigma-\mu}}
       \|\vec{J}\|_{\vec{p},\sigma}\delta^{\frac{\sigma-\mu}{\mu+4}}\notag\\
    &+\left( \frac{2\lambda}{\sqrt{\tau}}+(2\tau)^{\mu+7/2}C_6\right)\|\vec{J}\|_{\vec{p},0}
    \delta^{\frac{1}{2(\mu+4)}},\notag%\label{eq:mainerr1}
\end{align}
which completes the proof of the theorem.
\end{proof}
%---------------------------------------------------------------------------------------------------------------
%---------------------------------------------------------------------------------------------------------------
\section{Numerical simulations and discussions}

In this section, we present several numerical examples to demonstrate the feasibility and effectiveness of the proposed method.

In order to generate data sets, we solve the forward problem via the direct integration. For ${\rm{supp}}\vec{J}\subset V_0$, the unique solution to equation $\eqref{eq:maxwell}$ with the radiation condition $\eqref{eq:radiation}$ is given by $\eqref{eq:M_solve}$. We use the Gauss quadrature to calculate these volume integrals over the $48^3$ Gauss-Legendre points.

Now we specify details of the numerical implementation of the Fourier method. Let $V_0=[-0.5,0.5]^3$ and the radiated fields be measured on the unit sphere $\mathbb{S}^2$, i.e., $R=1$. %that is their polar angles are uniform
%We aim to reconstruct the true source $\vec{J}(x),x\in V_0$ by the Fourier expansion $\vec{J}_N^\delta(x),x\in V_0$.
In order to test the stability of the method, we add random noise to the data set. Since
\[
    \hat{\vec{x}}\times \vec{E}=( \hat{\vec{x}}\times \vec{E}\cdot \vec{e}_{\theta})\vec{e}_{\theta}+(\hat{\vec{x}}\times \vec{E}\cdot \vec{e}_{\varphi})\vec{e}_{\varphi},
\]
then the noisy data was given by the following formula:
\[
    \hat{\vec{x}}\times \vec{E}^{\delta}=( \hat{\vec{x}}\times \vec{E}^{\delta}\cdot \vec{e}_{\theta})\vec{e}_{\theta}+
    (\hat{\vec{x}}\times \vec{E}^{\delta}\cdot \vec{e}_{\varphi})\vec{e}_{\varphi},
\]
where
\begin{align*}
    &\hat{\vec{x}}\times \vec{E}^{\delta}\cdot \vec{e}_{\theta} =\hat{\vec{x}}\times \vec{E}\cdot \vec{e}_{\theta}
    +\delta r_1|\hat{\vec{x}}\times \vec{E}\cdot \vec{e}_{\theta}|\me^{\mi\pi r_2},\\
    &\hat{\vec{x}}\times \vec{E}^{\delta}\cdot \vec{e}_{\varphi} =\hat{\vec{x}}\times \vec{E}\cdot \vec{e}_{\varphi}
    +\delta r_3|\hat{\vec{x}}\times \vec{E}\cdot \vec{e}_{\varphi}|\me^{\mi\pi r_4},
\end{align*}
$r_1,r_2,r_3$ and $r_4$ are uniformly distributed random numbers, which range from $-1$ to $1$, and $\delta>0$ is the noise level. In terms of Theorem \ref{th:mainerr}, if not otherwise specified, we choose the truncation order $N$ by the following rule
\begin{equation}\label{eq:NNN}
        N:=[3\delta^{-1/4}]+1,
\end{equation}
and let
\[
    \mathbb{K}_N=\{ 2\pi|\vec{l}|\mid \vec{l}\in \mathbb{Z}^3,1\leq |\vec{l}|\leq N\},\quad
    k^*=2\pi\lambda,\quad \lambda=10^{-2}
\]
be the set of admissible wavenumbers. The radiated data are measured on $8\times 10^4$ uniformly distributed observation points located on the unit sphere $\mathbb{S}^2$. Using these radiated data and the formulae \eqref{eq:cal1}--\eqref{eq:cal4}, we compute the artificial data $\hat{\vec{x}}\times \vec{E}^{\delta}$ and $\hat{\vec{x}}\times\nabla\times \vec{E}^{\delta}$ on the sphere centered at the origin with radius of $\rho=1.2$ for the admissible wavenumbers $\mathbb{K}_N\cup \{k^*\}$. The surface integrals in $\eqref{eq:cal1}$ and $\eqref{eq:cal2}$ were evaluated using the trapezoidal rule. The series $\eqref{eq:cal3}$ and $\eqref{eq:cal4}$ were numerically truncated by $|n|\leq 20$. Finally, the Fourier coefficients $a_{\vec{l}}^{\delta}$, $b_{\vec{l}}^{\delta}$, $1\leq|\vec{l}|\leq N$ and $a_{\vec{0}}^{N,\delta}$ were computed by $\eqref{eq:al_form_new}$, $\eqref{eq:bl_form_new}$ and $\eqref{eq:a0_form_new}$, respectively, and the surface integrals here were also calculated using the trapezoidal rule over a $200\times 400$ grid uniformly points on $\Gamma_{\rho}$.

Numerically, the relative $L^2$ error $\|\vec{J}-\vec{J}_N^{\delta}\|_{(L^2(V_0))^3}/\|\vec{J}\|_{(L^2(V_0))^3}$ has the following discrete form
\begin{equation}
\frac{\left(\sum_{m=1}^{101^3}|\vec{J}(\vec{x}_m)-\vec{J}_N^{\delta}(\vec{x}_m)|^2\right)^{1/2}}
{\left(\sum_{m=1}^{101^3}|\vec{J}(\vec{x}_m)|^2\right)^{1/2}}.
\end{equation}
Here, $\vec{x}_m\in V_0,\ m=1,2,\cdots,101^3$ are uniformly spaced points and the pointwise values $\vec{J}_N^{\delta}(\vec{x}_m),\ m=1,2,\cdots,101^3$ were computed by $\eqref{eq:Jnd}$.

\begin{remark}
We would like to point out that the discrete relative $L^2$ error is insensitive to moderate choices of parameters $\lambda$ and $\rho$. In our experience, our experiments show that $\lambda=10^{-2}, 10^{-3}, 10^{-4}$ and $\rho=1.2, 1.5, 2$ would produce qualitatively the similar reconstructions.
\end{remark}

\noindent\textbf{Example 1.} Reconstruction of a smooth source function. In this example, we aim to reconstruct a smooth source function
\[
    \vec{J}_1=\vec{p}_1 f_1+\vec{p}_1\times \nabla g_1
\]
with $\vec{p}_1=(1,\sqrt{2},\sqrt{3})/\sqrt{6}$ and
\begin{align*}
    & f_1(x_1, x_2, x_3)=\sqrt{6}\exp\left(-80\left((x_1-0.15)^2+(x_2-0.15)^2+x_3^2\right)\right), \\
    & g_1(x_1, x_2, x_3)=\frac{\sqrt{6}}{10}\exp\left(-40(x_1^2+x_2^2+x_3^2)\right).
\end{align*}

Figure \ref{Fig:1} gives the second component of the reconstruction $\vec{J}_{1,N}^{\delta}$ for the source $\vec{J}_1$ with $N=6$ and $\delta=10\%$. For comparison, we give some quantitative results in Table \ref{tab:1} and Table \ref{tab:2}. In Table 1, we list the truncation orders $N(\delta)$ and relative $L^2$ errors of the reconstructions of $\vec{J}_1$ for different noise levels $\delta$.
Table 1 shows that as the noise level $\delta$ decreases, the corresponding truncation order $N(\delta)$ chosen by (\ref{eq:NNN}) increases, meanwhile the relative $L^2$ error decreases. Table \ref{tab:1} also illustrates that our method is insensitive to pollutions of the measured data.

To test the influences of different choices of $N$ and $\delta$ other than \eqref{eq:NNN}, we list the relative $L^2$ errors of the reconstructions of $\vec{J}_1$ for different noise levels $\delta$ and truncation orders $N$ in Table \ref{tab:2}.
As is shown in Table \ref{tab:2}, for some fixed noise level $\delta$, the error decreases as the truncation order $N$ increases, and for some fixed truncation order $N$, the error decreases as the noise level $\delta$ decreases.

\begin{figure}
\centering
\subfloat[]{\includegraphics[scale=0.4]{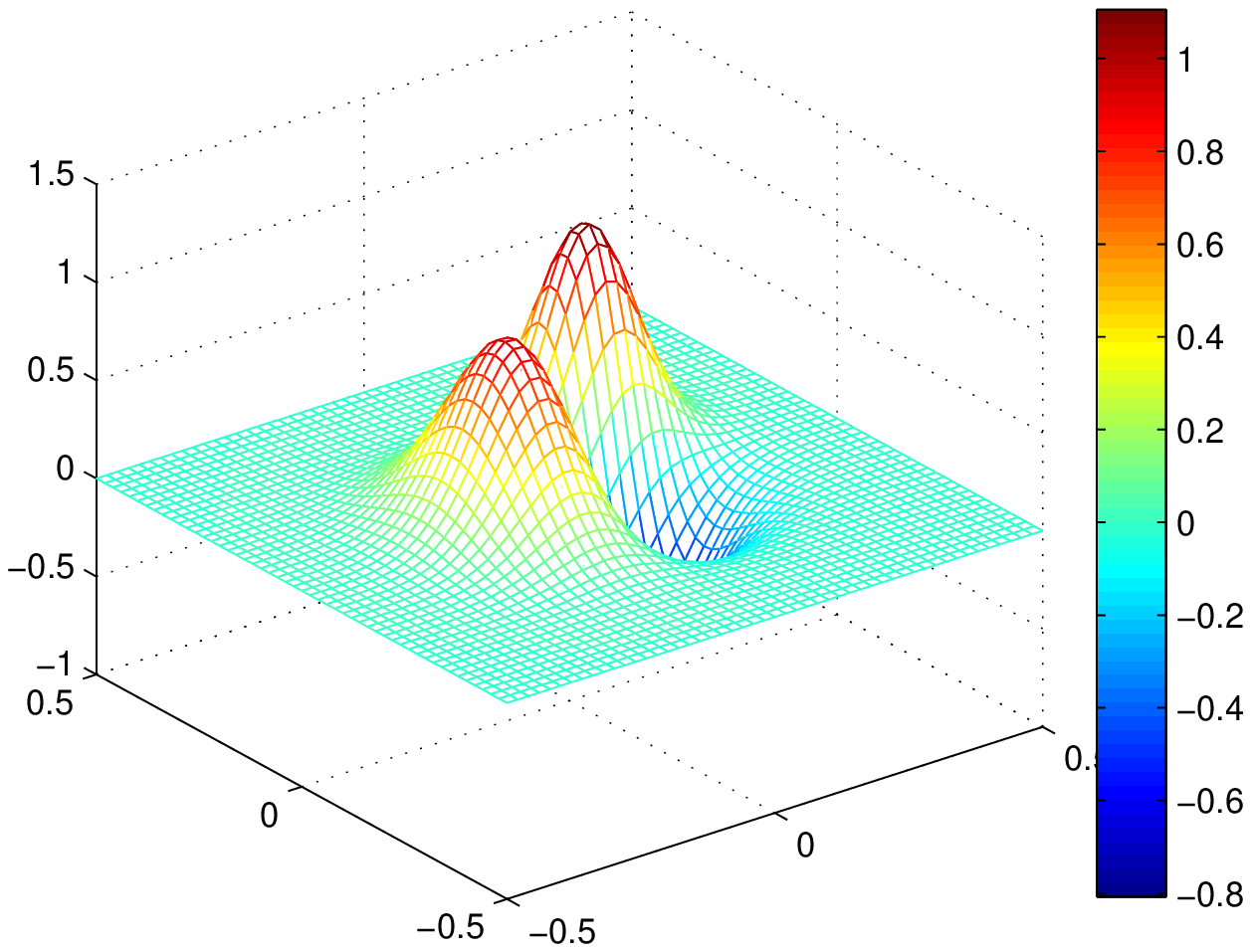}}
\subfloat[]{\includegraphics[scale=0.4]{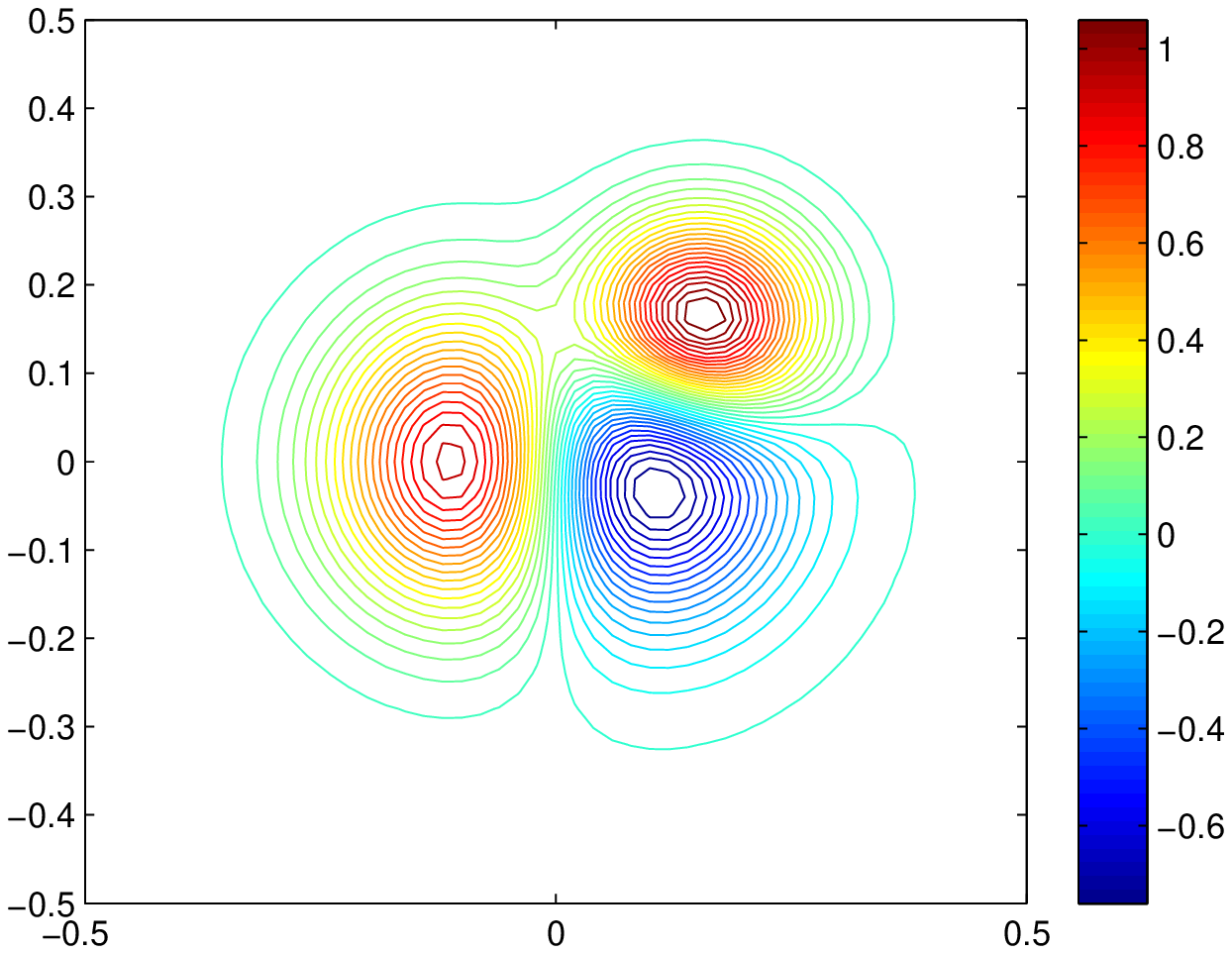}}\\
\subfloat[]{\includegraphics[scale=0.4]{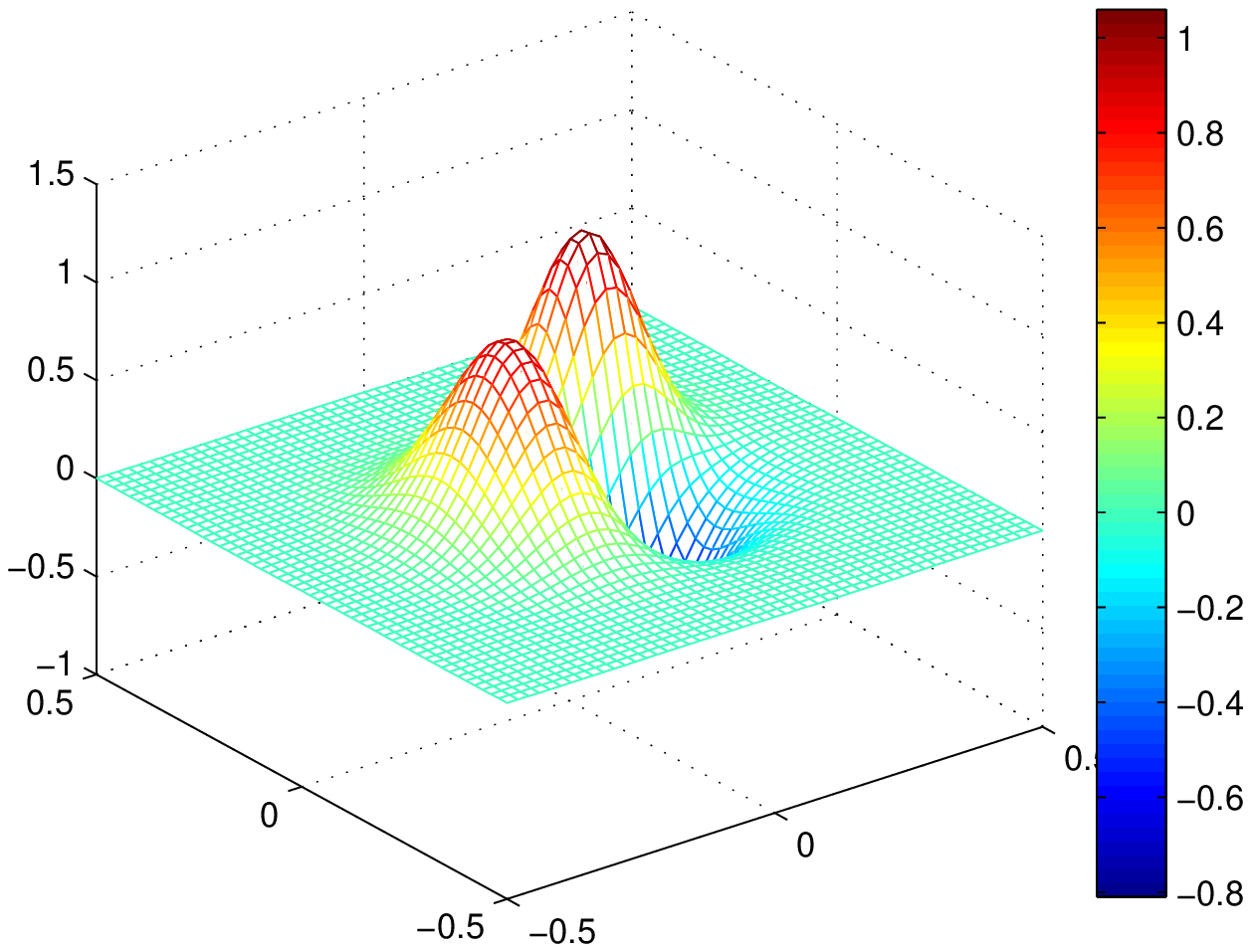}}
\subfloat[]{\includegraphics[scale=0.4]{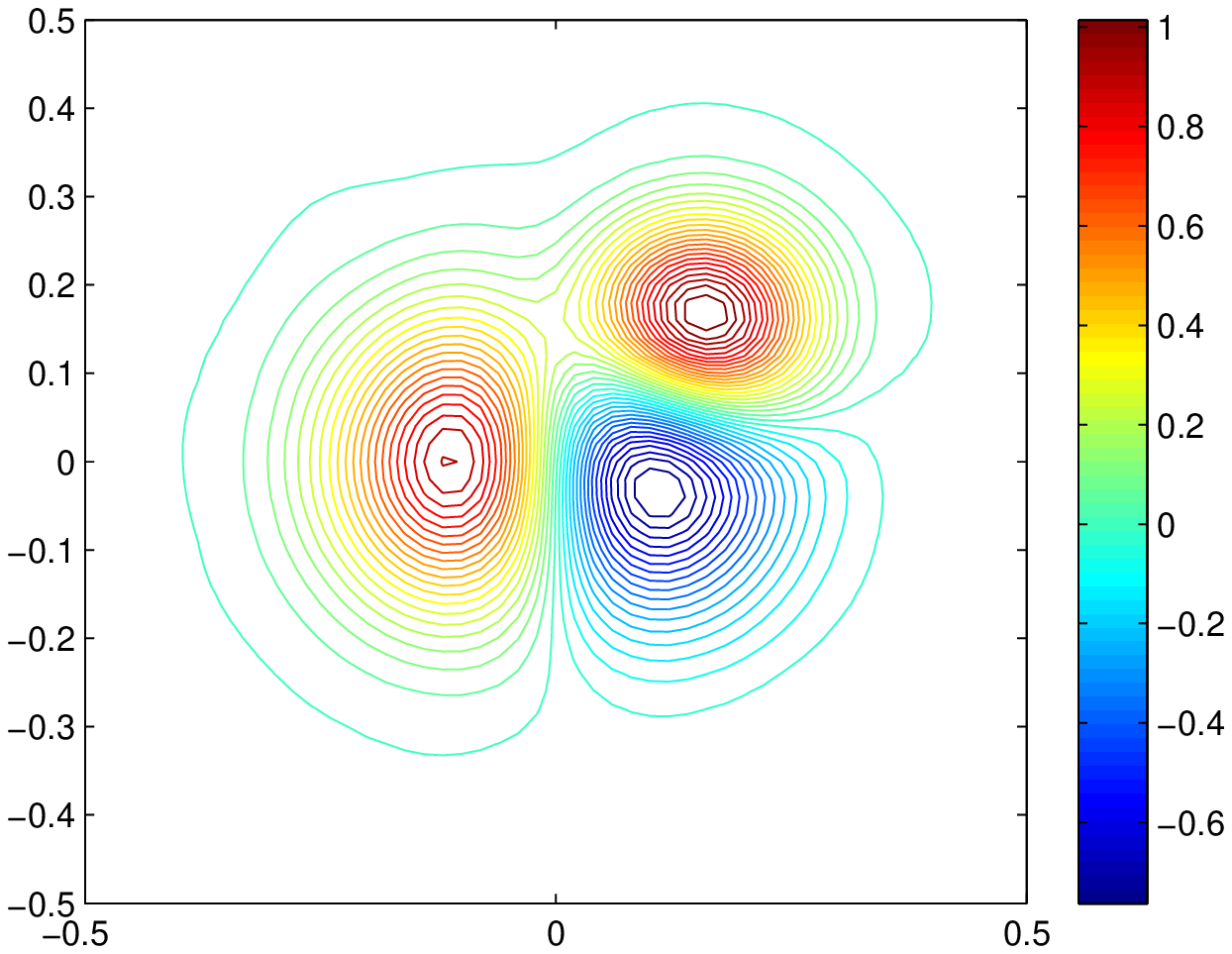}}
\caption{The second component of the reconstruction of $\vec{J}_1$ with $N=6$ and $\delta=10\%$. (a) exact, surface plot (b) exact, contour plot (c) reconstructed, surface plot (d) reconstructed, contour plot.}
\label{Fig:1}
\end{figure}

\begin{table}
\centering
\caption{The truncation orders $N(\delta)$ and relative $L^2$ errors of the reconstructions of $\vec{J}_1$ for different noise levels $\delta$.}
\label{tab:1}
\begin{tabular}{lccccc}
  \toprule
  $\delta$      & $1\%$       & $2\%$          & $5\%$          & $10\%$  \\
  \midrule
  $N(\delta)$ & 10             & 8                   & 7                  & 6       \\
%  error           & $1.137\%$ & $1.1418\%$ & $1.1901\%$ & $1.6182\%$ \\
  error       & $1.141\%$      & $1.146\%$           & $1.198\%$          & $1.692\%$ \\
  \bottomrule
\end{tabular}
\end{table}

\begin{table}
\centering
\caption{The relative $L^2$ errors of the reconstructions of $\vec{J}_1$ for different noise levels $\delta$ and truncation orders $N$.}\label{tab:2}
\begin{tabular}{lcccccc}
  \toprule
   & \multicolumn{6}{c}{\underline{\hspace{4cm} $N$ \hspace{4cm}}}\\
              & 5         & 6        & 7        & 8         & 9         & 10 \\
  \midrule
  $\delta=1\%$  & $4.842\%$  & $1.624\%$ & $1.174\%$ & $1.142\%$ & $1.141\%$  & $1.141\%$ \\
  $\delta=2\%$  & $4.843\%$  & $1.628\%$ & $1.179\%$ & $1.146\%$ & $1.145\%$  & $1.145\%$ \\
  $\delta=5\%$  & $4.848\%$  & $1.642\%$ & $1.198\%$ & $1.166\%$ & $1.165\%$  & $1.165\%$ \\
  $\delta=10\%$ & $4.865\%$  & $1.692\%$ & $1.265\%$ & $1.235\%$ & $1.234\%$  & $1.234\%$ \\
% $\delta=1\%$ & $4.811\%$  & $1.617\%$ & $1.171\%$ & $1.138\%$  & $1.137\%$  & $1.137\%$ \\
% $\delta=5\%$ & $4.816\%$ & $1.632\%$ & $1.190\%$ & $1.169\%$ & $1.168\%$ & $1.168\%$ \\
% $\delta=10\%$  & $4.834\%$  & $1.618\%$ & $1.275\%$  & $1.230\%$ & $1.229\%$ & $1.245\%$\\
%  $\delta=20\%$  & $4.9105\%$   & $1.8933\%$ & $1.5288\%$ & $1.5047\%$  & $1.504\%$  & $1.504\%$ \\
  \bottomrule
\end{tabular}
\end{table}

\noindent\textbf{Example 2.} Reconstruction of a discontinuous source function. In this example, we reconstruct a discontinuous source function defined in the cubic domain $V_0$ by
\[
    \vec{J}_2=\vec{p}_2 f_2
\]
where $\vec{p}_2=(1,\sqrt{2},\sqrt{3})^{\top}/\sqrt{6}$ and
\[
    f_2(x_1, x_2, x_3)=
    \begin{cases}
    \sqrt{6},     &\text{if} \ 0\leq x_1,x_2\leq 0.4, -0.2\leq x_3\leq 0.2,\\
    \sqrt{6}/2,   &\text{if} \ (x_1+0.25)^2+(x_2+0.25)^2+x_3^2\leq 0.15^2,\\
    0,     &\text{elsewhere}.
    \end{cases}
\]

The support of $\vec{J}_2$ consists of a cubic and a sphere. Figure 3 gives the second component of the reconstruction $\vec{J}_{2,N}^{\delta}$ for the source $\vec{J}_2$ with different noise levels $\delta$.

Figure \ref{fig:2} shows that as the noise level $\delta$ decreases (the number of Fourier terms $N(\delta)$ increases correspondingly), the error of the reconstruction is reduced. In addition, Gibbs phenomenon in the theory of Fourier series can be observed at the discontinuous points of the source function.

\begin{figure}
\centering
\subfloat[]{\includegraphics[scale=0.40]{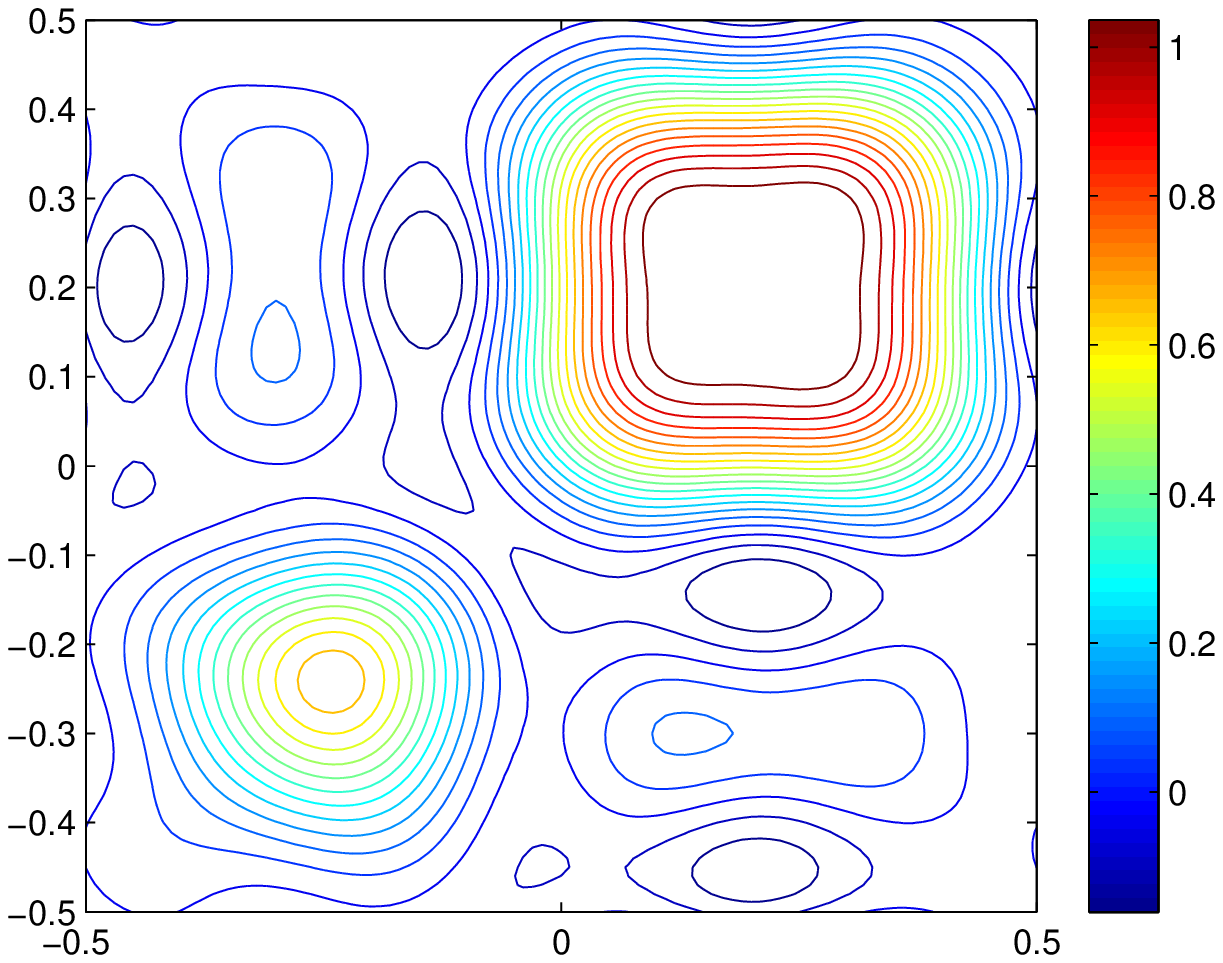}}
\subfloat[]{\includegraphics[scale=0.40]{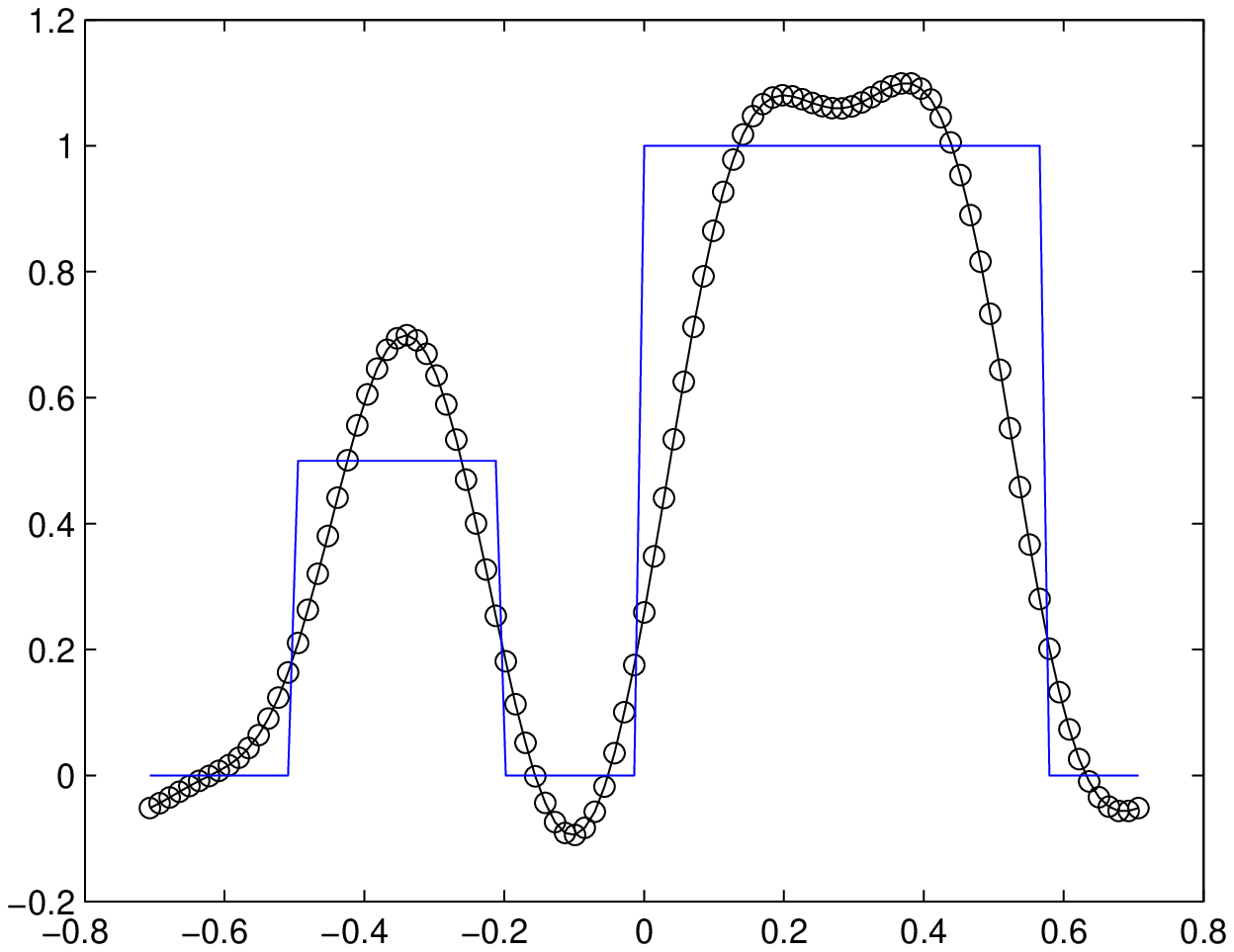}} \\
\subfloat[]{\includegraphics[scale=0.40]{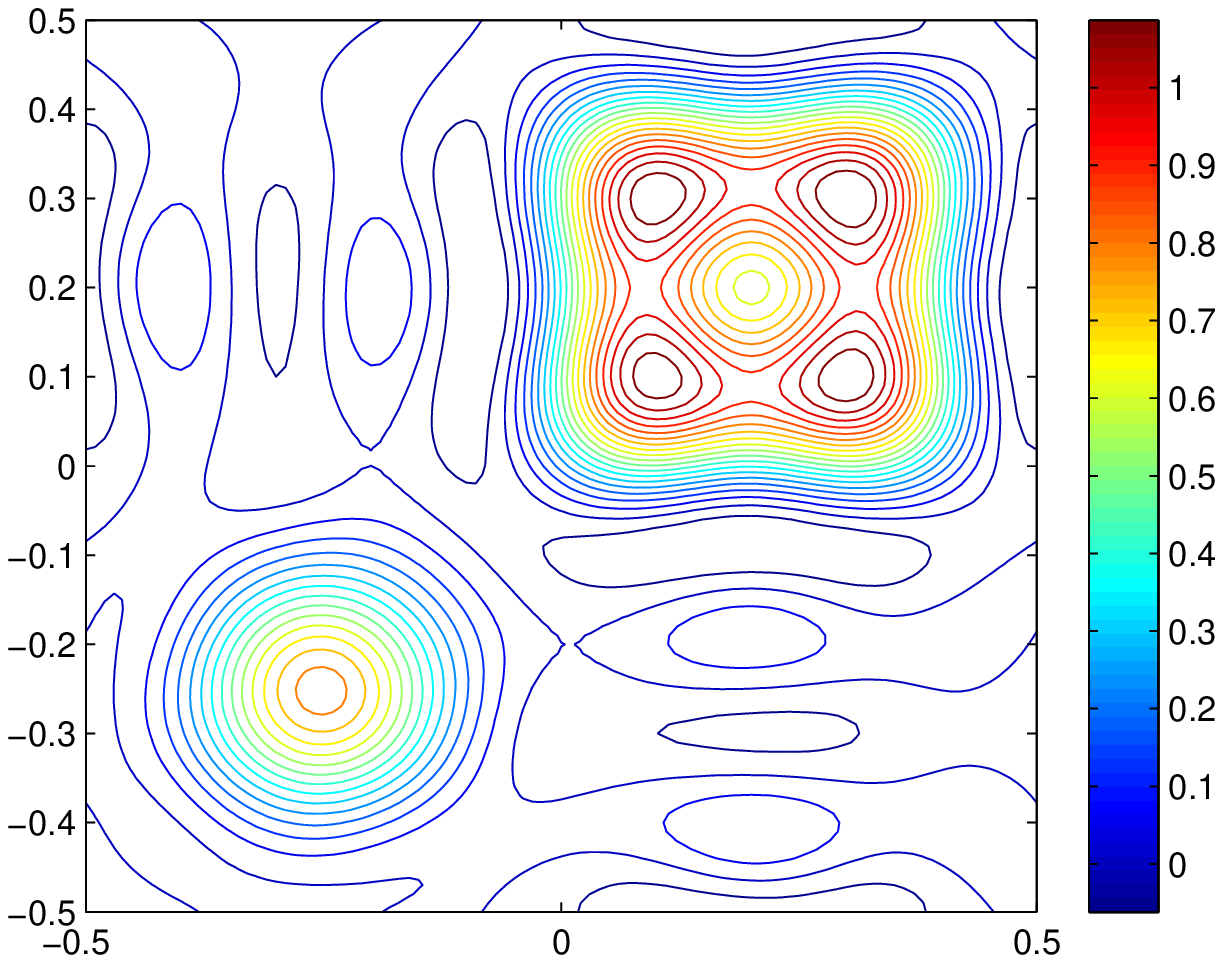}}
\subfloat[]{\includegraphics[scale=0.40]{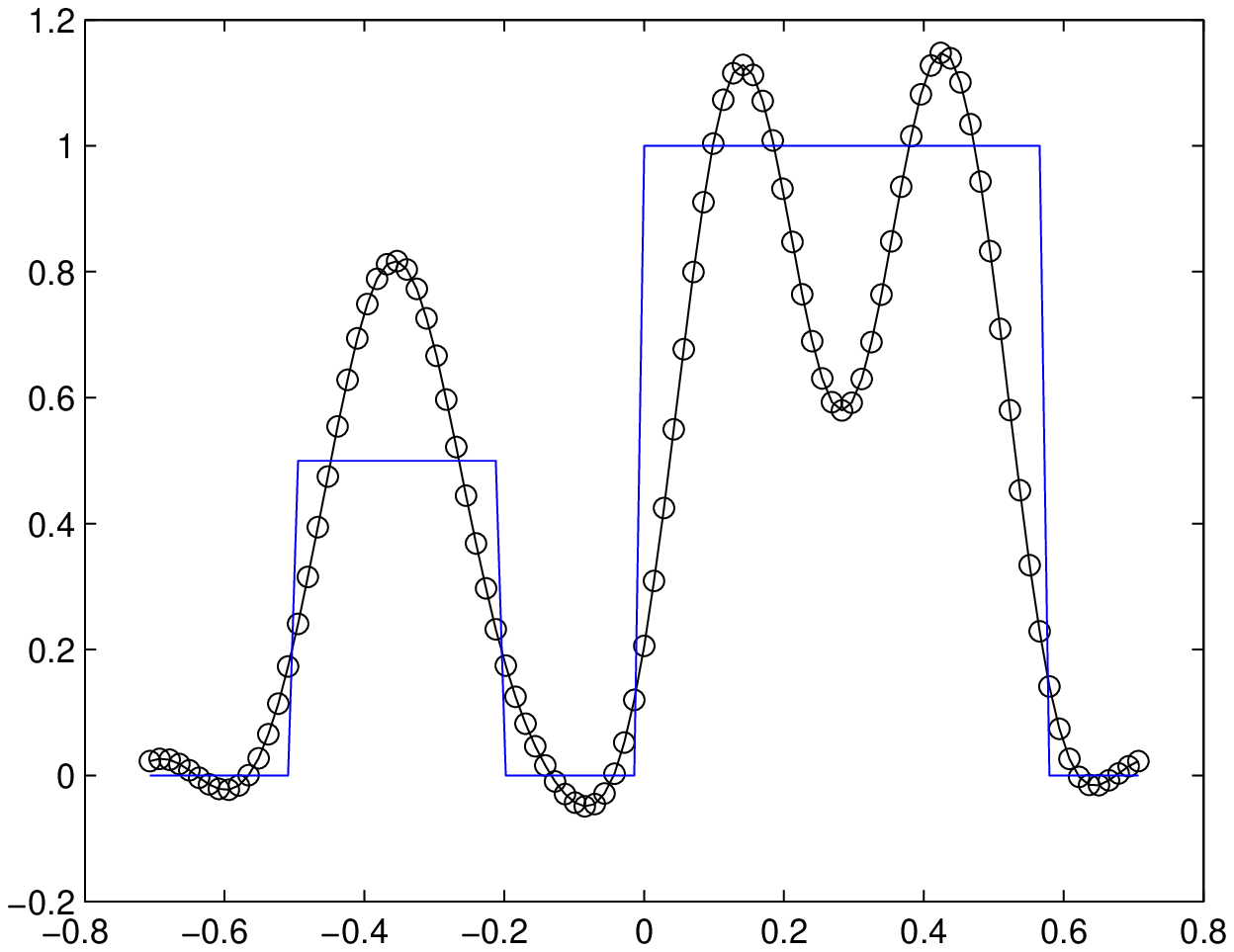}} \\
\subfloat[]{\includegraphics[scale=0.40]{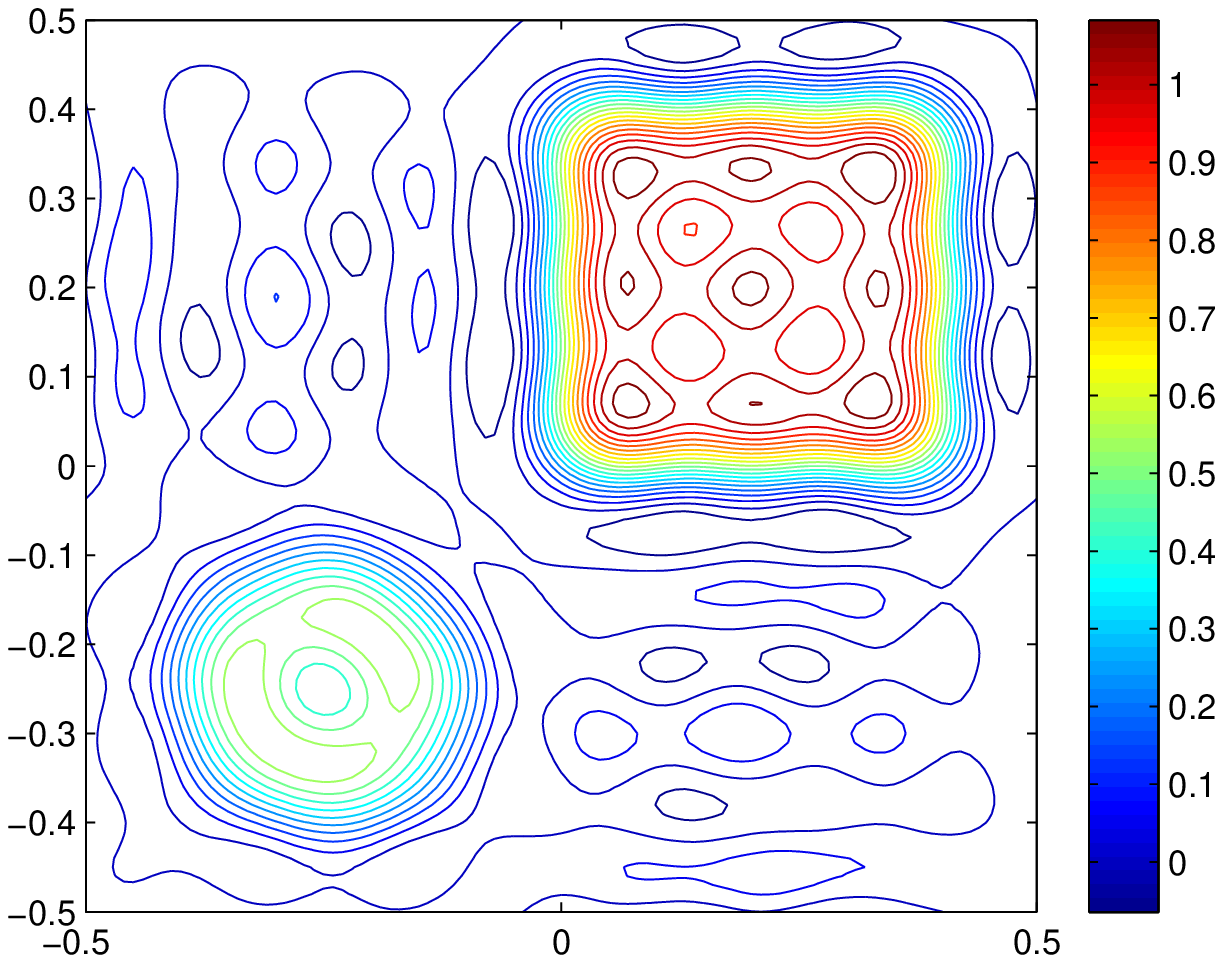}}
\subfloat[]{\includegraphics[scale=0.40]{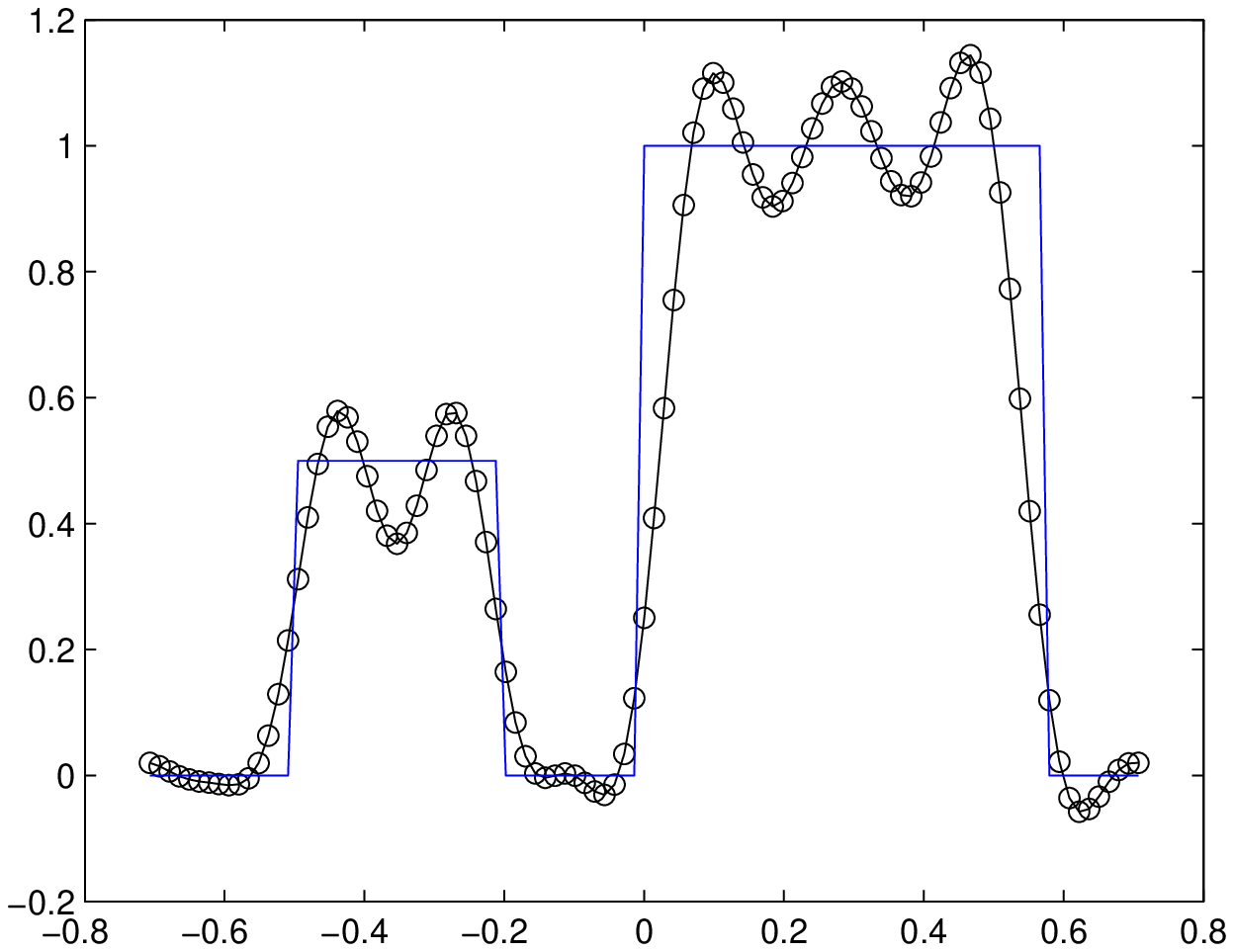}}
\caption{The reconstruction of the source $\vec{J}_2$ with different noise levels $\delta$. The cross section plots are depicted at $x_3=0$ and $x_1=x_2$.
(a) $\delta=0.1$, contour plot;
(b) $\delta=0.1$, cross section plot;
(c) $\delta=0.05$, contour plot;
(d) $\delta=0.05$, cross section plot;
(e) $\delta=0.01$, contour plot;
(f) $\delta=0.01$, cross section plot. }
\label{fig:2}
\end{figure}

\noindent\textbf{Example 3.} Finally, we aim to test the method for a different polarization direction. In this example,  we are going to reconstruct the source function of the form
\[
    \vec{J}_3=\vec{p}_3 f_3+\vec{p}_3\times \nabla g_3
\]
where $\vec{p}_3=(\sqrt{5},-1,\sqrt{3})^{\top}/3$, and
\begin{align*}
  & f_3(x_1, x_2, x_3)=3\exp(-80((x_1-0.15)^2+(x_2-0.15)^2+x_3^2)), \\
  & g_3(x_1, x_2, x_3)=0.3\exp(-40(x_1^2+x_2^2+x_3^2)).
\end{align*}

Figure \ref{fig:3} shows the second component of the reconstruction $\vec{J}_{3,N}^{\delta}$ for the source $\vec{J}_3$ with  $N=6$ and $\delta=0.1$.
\begin{figure}
\centering
\subfloat[]{\includegraphics[scale=0.40]{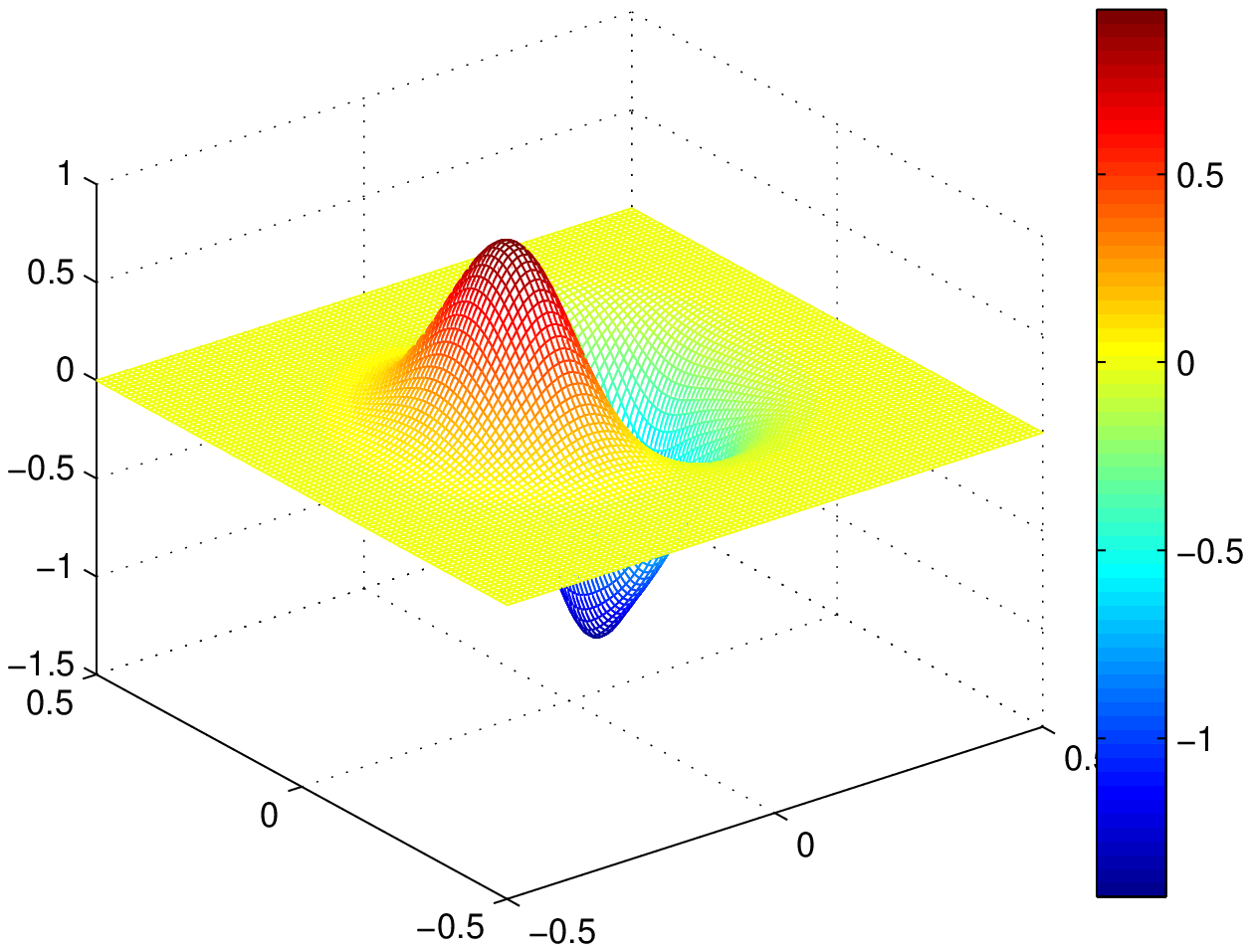}}
\subfloat[]{\includegraphics[scale=0.40]{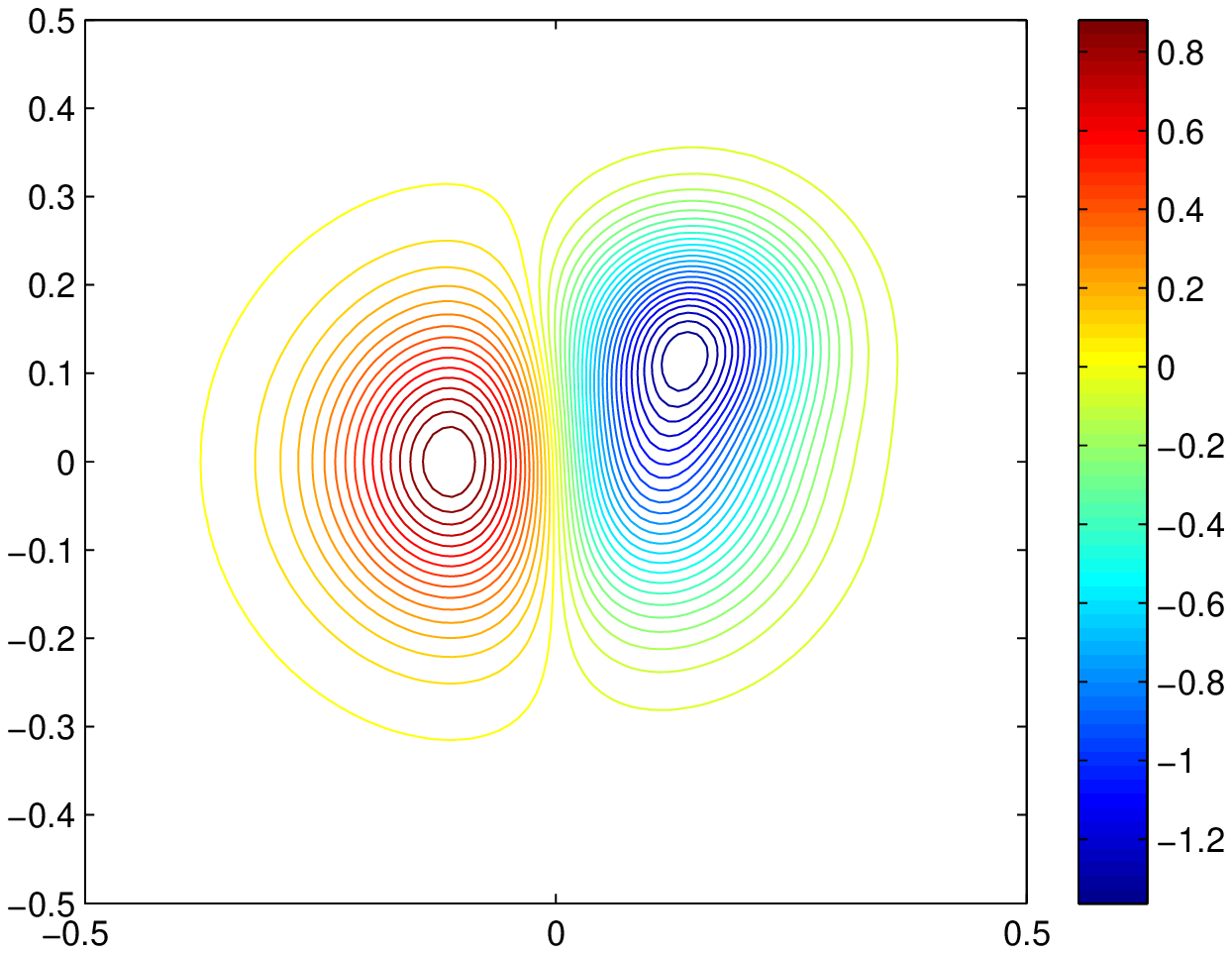}} \\
\subfloat[]{\includegraphics[scale=0.40]{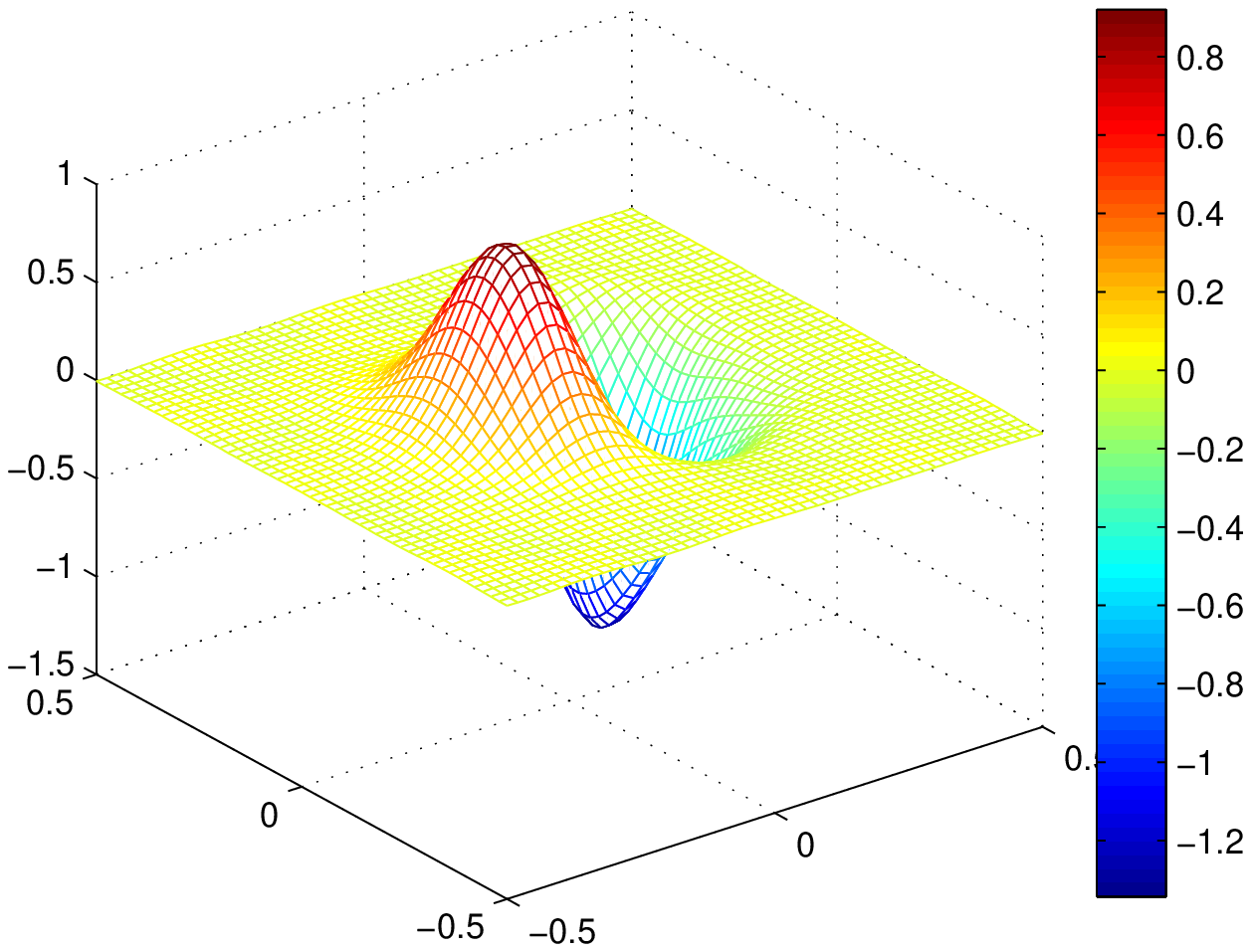}}
\subfloat[]{\includegraphics[scale=0.40]{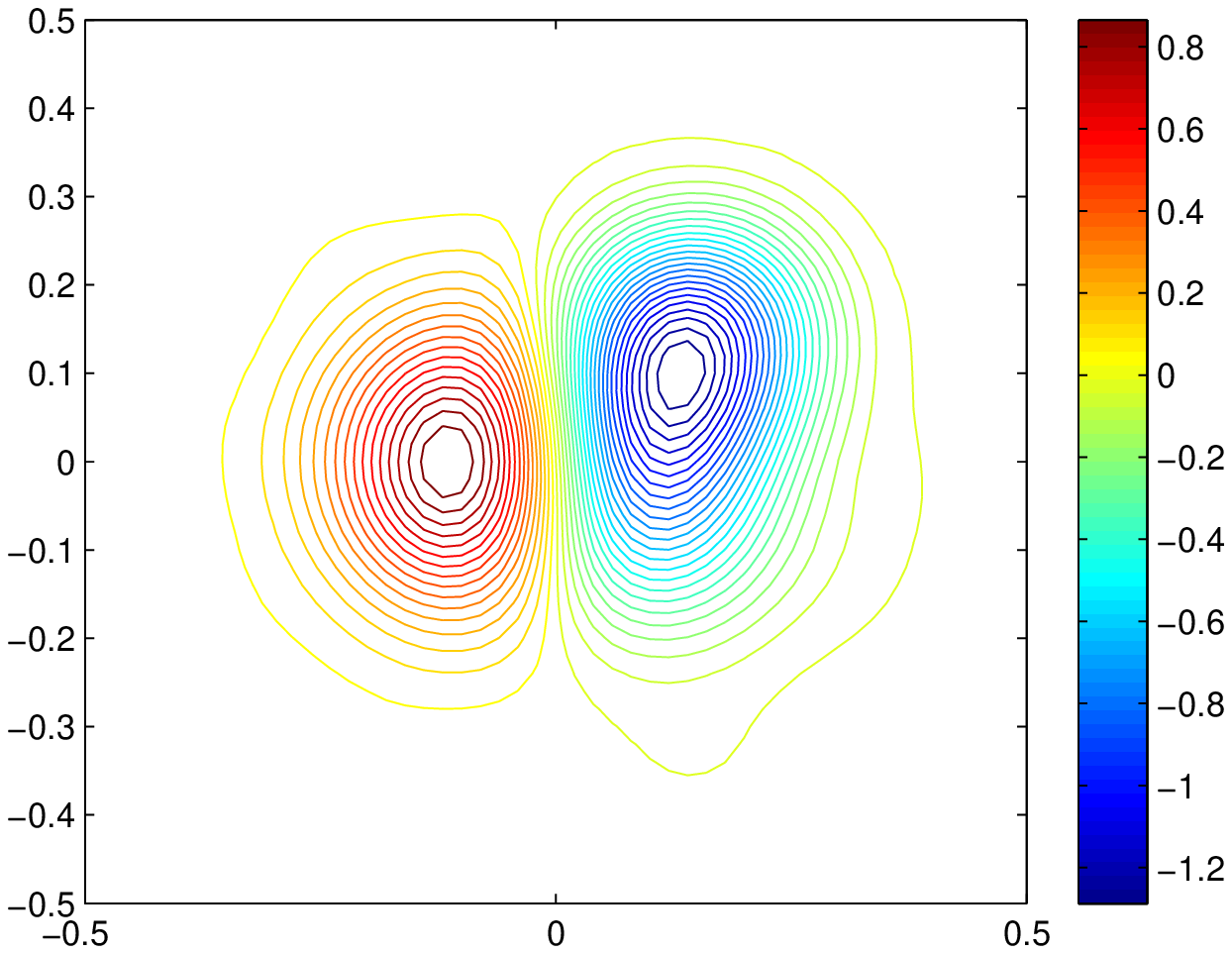}}
\caption{The second component of the reconstruction of $\vec{J}_3$ with $N = 6$ and
$\delta=0.1$. (a) exact, surface plot; (b) exact, contour plot; (c) reconstructed, surface plot; (d) reconstructed, contour plot.}
\label{fig:3}
\end{figure}

\section*{Acknowledgements}
 The work of G. Wang and F. Ma were supported by the NSF grant of China under 11371172. The work of Y. Guo was supported by the NSF grants of China under 11601107, 11671111 and 41474102. The work of J. Li was supported by the NSF grant of China under 11571161, the Shenzhen Sci-Tech Fund under JCYJ20160530184212170 and the SUSTech Startup fund. The authors would also wish to thank Prof. Deyue Zhang for his suggestions that improved the quality of this paper.

%The authors would also like to thank the anonymous referees for many constructive comments and suggestions.

\end{document}